\documentclass[american]{amsart}
\usepackage{amssymb, amsmath, amsfonts, amsthm, esint, mathtools}
\usepackage[usenames,dvipsnames,svgnames,table]{xcolor}
\usepackage[margin=1in]{geometry} 
\usepackage[utf8]{inputenc}
\usepackage{tikz} 
\usepackage[parfill]{parskip}
\usepackage{enumerate}
\usepackage{bbm}

\usepackage{cite}
\usepackage[colorlinks,linkcolor = blue,citecolor = green]{hyperref}

\usepackage{comment}

\mathtoolsset{showonlyrefs}

\newcommand{\abs}[1]{\left\vert#1\right\vert}
\newcommand{\Abs}[1]{\big\vert#1\big\vert}
\newcommand{\norm}[1]{\left\|#1\right\|}  
\newcommand{\set}[1]{\left\{ #1 \right\}}
\newcommand{\brak}[1]{\left\langle #1 \right\rangle}

\newcommand{\dd}{\, {\rm d}}
\newcommand{\R}{\ensuremath{{\mathbb R}}}
\newcommand{\weak}{\ensuremath{{\rightharpoonup}}}
\renewcommand{\S}{\ensuremath{{\mathcal S}}}
\newcommand{\N}{\ensuremath{{\mathbb N}}}

\DeclareMathOperator{\eps}{\varepsilon}
\DeclareMathOperator{\embeds}{\hookrightarrow}

\newcommand{\beq}{\begin{equation}}
\newcommand{\eeq}{\end{equation}}
\newcommand{\beqs}{\begin{equation*}}
\newcommand{\eeqs}{\end{equation*}}
\newcommand{\bal}{\begin{equation}\begin{aligned}}
\newcommand{\eal}{\end{aligned}\end{equation}}
\newcommand{\bals}{\begin{equation*}\begin{aligned}}
\newcommand{\eals}{\end{aligned}\end{equation*}}

\newcounter{num} \numberwithin{num}{section}

\newtheorem{theorem}[num]{Theorem}
\newtheorem{proposition}[num]{Proposition}
\newtheorem{lemma}[num]{Lemma}
\newtheorem{corollary}[num]{Corollary}
\theoremstyle{definition}

\theoremstyle{remark}

\numberwithin{equation}{section}

\author{William Golding, Am\'elie Loher}
\title[Strong solutions of the Landau-Coulomb equation]{Local-in-time strong solutions of the homogeneous Landau-Coulomb equation with $L^p$ initial datum}
\date{\today}							
\address[William Golding]{\newline Department of Mathematics, \newline The University of Texas at Austin, Austin, TX 78712, USA}
\email{wgolding@utexas.edu}
\address[Amélie Loher]{\newline Department of Pure Mathematics and Mathematical Statistics, \newline University of Cambridge, Cambridge, CB3 0WA, United Kingdom}
\email{ajl221@cam.ac.uk}

\thanks{\textbf{Acknowledgments:} The authors would like to thank Maria Gualdani for her contribution in proposing this problem, for fostering our collaboration, and for engaging in several insightful discussions.}
\thanks{\textbf{Funding:} W. Golding is partially supported by NSF grant DMS 1840314. A. Loher is funded by the Cambridge Trust and Newnham College scholarship.}

\begin{document}

\begin{abstract}
    We consider the homogeneous Landau equation with Coulomb potential and general initial data $f_{in} \in L^p$, where $p$ is arbitrarily close to $3/2$. We show the local-in-time existence and uniqueness of smooth solutions for such initial data. The constraint $p > 3/2$ has appeared in several related works and appears to be the minimal integrability assumption achievable with current techniques. We adapt recent ODE methods and conditional regularity results appearing in [arXiv:2303.02281] to deduce new short time $L^p \to L^\infty$ smoothing estimates. These estimates enable us to construct local-in-time smooth solutions for large $L^p$ initial data, and allow us to show directly conditional regularity results for solutions verifying \emph{unweighted} Prodi-Serrin type conditions. As a consequence, we obtain additional stability and uniqueness results for the solutions we construct.
\end{abstract}

\maketitle

\tableofcontents

\section{Introduction}
The Landau-Coulomb equation is one of the fundamental models of plasma physics that models the statistical evolution of a collisional plasma. In this manuscript, we consider the spatially homogeneous Landau-Coulomb equation, written as
\begin{equation}
	\partial_t f(t, v) = \mathcal{Q}(f,f)(t, v) \qquad (t, v) \in (0, \infty) \times \R^3,
\label{eq:landau1}
\end{equation}
where $t$ is the time variable, $v$ is the velocity variable, $f:\R^+ \times \R^3 \rightarrow \R^+$ is the unknown distribution function, and $\mathcal Q$ is the collision operator, given by 
\begin{equation}\label{eq:q}
	\mathcal Q(f,f)(v) := \frac{1}{8\pi}\nabla_v \cdot \left(\int_{\R^3} \frac{\Pi(v - v_*)}{\abs{v-v_*}} \left\{f_*\nabla f - f \nabla f_*\right\}\dd v_*\right), \qquad \Pi(z) := \rm{Id} - \frac{z \otimes z}{\abs z^2},
\end{equation}
with the standard notation $f = f(v), ~ f_* = f(v_*)$. While \eqref{eq:landau1} is not as physically relevant as the full inhomogeneous model, the study of \eqref{eq:landau1} already provides several mathematical challenges, and is a reasonable starting point for the study of the full inhomogeneous equation. Indeed, \eqref{eq:landau1} may be rewritten as the following quasi-linear, divergence-form parabolic equation:
\beq	
	\partial_t f = \nabla \cdot \left(A[f]\nabla f - \nabla a[f] f\right),
\label{eq:landau}
\eeq
where the nonlocal coefficients $A$ and $a$ are given by convolutions:
\begin{equation}\label{eq:coefficients}
A[f] = \frac{\Pi(v)}{8\pi |v|} \ast f, \qquad a[f] = (-\Delta)^{-1}f = \frac{1}{4\pi |v|} \ast f.
\end{equation}
It is also possible to rewrite \eqref{eq:landau} in non-divergence form as:
\begin{equation}\label{eq:landau-nondiv}
\partial_t f = A[f]:\nabla^2 f + f^2,
\end{equation}
which highlights one major challenge in studying this equation, namely the competition between the reaction term and the nonlinear, non-local diffusion term. Indeed, written in non-divergence form, \eqref{eq:landau-nondiv} suggests a comparison to the semi-linear heat equation, which is known to blow-up in finite time.

We recall that \eqref{eq:landau} formally conserves mass, momentum, and energy:
\begin{equation}\label{eqn:conservation}
    \frac{\dd}{\dd t}\int_{\R^3} \begin{pmatrix} 1 \\ v \\ |v|^2 \end{pmatrix} f(t,v) \dd v = 0.
\end{equation}
Moreover, solutions to \eqref{eq:landau} have decreasing Boltzmann entropy:
\begin{equation}\label{eqn:Htheorem}
    \frac{\dd}{\dd t}\int_{\R^3} f \log(f) \dd v \le 0.
\end{equation}
Consequently, the steady states of \eqref{eq:landau} may be characterized explicitly: $f \in L^1_2$ is a steady state of \eqref{eq:landau} if and only if $f$ is a Maxwellian distribution, given by
\beqs
    \mu_{\rho, u, \theta}(v) = \frac{\rho}{(2\theta\pi)^{3/2}}e^{-\frac{\abs{v-u}^2}{2\theta}},
\eeqs
where $\rho \in \R^+$ is the total mass, $u \in \R^3$ is the mean velocity, and $\theta \in \R^+$ is the temperature of the plasma. Additionally, the conserved quantities and decreasing entropy suggest $L^1_2 \cap L\log(L)$ is a natural functional space for the problem. Indeed, in this setting, the existence of very weak solutions to \eqref{eq:landau1}, typically termed H-solutions, was shown in \cite{Villani_Hsols}. These H-solutions satisfy the natural estimates \eqref{eqn:conservation} and \eqref{eqn:Htheorem}. Later, it was shown that $H$-solutions belong to $L^1_{loc}(\R^+;L^3_{-9})$ (recently improved to $L^1_{loc}(\R^+;L^3_{-5})$ by Ji in \cite{Ji}) and, in particular, solve \eqref{eq:landau1} in an appropriate distributional sense \cite{Desvillettes_HSolutions}. These solutions were then shown to propagate $L^1$ moments of all orders and stretched exponential moments in \cite{CarrapatosoDesvillettesHe}, enabling the authors to obtain $L^1$-type convergence to equilibrium results for the Coulomb potential case considered here. Despite recent progress on the eventual regularity and partial regularity of $H$-solutions \cite{DesvillettesHeJiang,GolseGualdaniImbertVasseur_PartialRegularity1,GolseImbertVasseur_PartialRegularity2}, regularity and uniqueness of $H$-solutions remains open.

On the other hand, if one starts with more regular, or more integrable initial datum, the situation is reversed. Uniqueness for bounded solutions, i.e. belonging to $L^1_{loc}(\R^+;L^\infty)$ was shown by Fournier in \cite{Fournier} via a probabilistic method. 
Moreover, applying the parabolic De Giorgi-Nash-Moser theorem and parabolic Schauder estimates appropriately, it is known that $L^\infty_{loc}(0,T;L^\infty(\R^3))$ strong solutions are instantaneously smooth. However, the existence theory of strong solutions is much less developed.

For large initial data\footnote{By which we mean far from the unique corresponding Maxwellian steady state}, until recently, the only result was that of \cite{ArsenevPeskov}: the authors show local-in-time existence of a strong solution for initial data in $L^\infty$. The $L^\infty$ norm is rather amenable to estimation using maximum principle type arguments, which neglect the helpful diffusive term in \eqref{eq:landau}. By contrast, it is significantly more difficult to close energy estimates in lower $L^p$ norms. Recently, in \cite{DesvillettesHeJiang}, the authors were able to prove a corresponding result for initial datum in a polynomially weighted $\dot{H}^1$ space. We mention also that existence of strong solutions for radial initial data belonging to $L^p$ with $p > 3/2$ is established in \cite{GualdaniGuillen1}. Finally, there are additional works in the rough data, inhomogeneous Landau-Coulomb setting, i.e. \cite{HendersonSnelson,HendersonSnelsonTarfulea,HendersonSnelsonTarfulea1} and the references therein, but the class of initial data considered is always more stringent than just $L^\infty$, which was addressed in \cite{ArsenevPeskov} for homogeneous initial data.

For small initial data, i.e. near equilibrium, there are more results. In \cite{CarrapatosoMischler}, global-in-time strong solutions are constructed for data close to equilibrium in a weighted $L^2$ sense, using spectral properties of the collision operator linearized about the Maxwellian. Additionally, in \cite{Guo, KimGuoHwang,DesvillettesHeJiang}, global strong solutions are constructed for data close to equilibrium in weighted $H^k$ for some $k \ge 8$, weighted $L^\infty$, and weighted $H^1$, respectively, using different methods. In \cite{GoldingGualdaniLoher}, the authors were able to relax these constraints to an unweighted $L^p$ notion of smallness, for $p > 3/2$.

The primary aim of this paper is to adapt the energy methods and smoothing estimates in \cite{GoldingGualdaniLoher} to show local existence of strong solutions for initial data merely in $L^p$, for any $p > 3/2$. As a secondary aim, we attempt to unify the smoothing estimates that exist within the literature, which are essential to our method. Accordingly, we deduce a uniqueness result for solutions in unweighted mixed space-time Lebesgue spaces, reminiscent of the classical Prodi-Serrin results in incompressible fluid mechanics.

It is rather natural to expect smoothing estimates for solutions of \eqref{eq:landau}, given the parabolic nature of the equation. This perspective has yielded several conditional results, beginning with \cite{GualdaniGuillen1}, which shows radial solutions regularize from $L^\infty_{loc}(\R^+;L^{\frac{3}{2}+})$ to $L^\infty_{t,v}$ and become instantaneously smooth, and \cite{Silvestre}, which extended this to non-radial solutions, at the expense of weights in velocity. This was reproved via new techniques in \cite{GualdaniGuillen2}, which interestingly suggests that the nonlinear diffusive effects of \eqref{eq:landau} are far stronger than the corresponding diffusive effects of the heat equation. Additionally, recent work in \cite{BenPorath} suggests some of these estimates can be localized in velocity. Finally, a framework of $L^\infty_{loc}(\R^+;L^{\frac{3}{2}+})$ plus some energy control appeared in \cite{AlonsoBaglandsDesvillettesLods} to obtain pointwise control for the closely related Landau-Fermi-Dirac equation. Subsequently, this technique was modified for the use of the Landau-Coulomb equation in \cite{GoldingGualdaniLoher} and reappeared in \cite{AlonsoBaglandDesvillettesLods_ProdiSerrin}, generalized to the full range of soft potentials and dimensions. We attempt to unite these smoothing results in Proposition \ref{prop:degiorgi} below, which is then used to build strong solutions.

\subsection{Main Results}
Before presenting our main result, we first establish some notation and revisit some basic properties of \eqref{eq:landau1}.
We recall that the formal conservation laws \eqref{eqn:conservation} imply mass, momentum, and energy are constant for any reasonable solution of \eqref{eq:landau}. Using that the Landau-Coulomb equation is translation invariant and has a two-parameter scaling invariance, we will assume that our initial datum $f_0(v)$ satisfies the normalization:
\begin{equation}\label{eq:normalization}
    \int_{\R^3} \begin{pmatrix} 1 \\ v \\ |v|^2\end{pmatrix} f_0(v) \dd v = \begin{pmatrix} 1 \\ 0 \\ 3\end{pmatrix}.
\end{equation}
This choice fixes the corresponding Maxwellian with same mass, momentum and energy as
\begin{equation*}
    \mu(v) := (2\pi)^{-\frac{3}{2}}e^{-\frac{\abs{v}^2}{2}}.
\end{equation*}
Throughout, we will use the Japanese bracket convention $\langle v\rangle := (1+\abs{v}^2)^{\frac{1}{2}}$ and employ the following convention to define polynomial weighted Lebesgue spaces: for $p \in [1, \infty)$ and $m \in \R$,
\begin{equation*}
    \norm{f}_{L^p_m(\R^3)}^p :=  \int_{\R^3} \abs{f(v)}^p\langle v\rangle^{m} \dd v, \qquad \norm{f}_{L^\infty_m(\R^3)} :=  \underset{v \in \R^3}{\mathrm{ess\,sup}}\, \langle v\rangle^{m}\abs{f(v)}.
\end{equation*}
In this notation, our main result reads as follows:
\begin{theorem}\label{thm:main}
Fix $p > \frac{3}{2}$ and let $m >\frac{9}{2}\frac{p-1}{p - \frac{3}{2}}$. For any given $M$, $H$ and $f_{in} \in L^1_m \cap L^p(\R^3)$ satisfying 
\begin{equation}\label{eq:bounds}
        \norm{f_{in}}_{L^1_m} \le M \qquad \text{and} \qquad \int_{\R^3} f_{in} \abs{\log(f_{in})} \;dv \le H,
\end{equation}
there exists $0 < T = T(p, m, M, H)$ such that equation \eqref{eq:landau} admits a smooth solution on $(0,T) \times \R^3$ with initial datum $f_{in}$ in the sense that $\lim_{t \rightarrow 0^+} \norm{f(t) - f_{in}}_{L^p} = 0$. This solution satisfies, for $t \in (0, T)$, the smoothing estimate
\begin{equation}\label{eq:inst_reg}
\|f(t)\|_{L^\infty} \le C(p,m,M,H)\left(1 + t^{-\beta^*}\right),
\end{equation}
where $\beta^* = \beta^*(p,m)$ is explicitly computable and satisfies $\frac{3}{2p} < \beta^* < 1$. Thus, $f$ is the unique solution to \eqref{eq:landau} in $L^1(0,T;L^\infty)$ with initial data $f_{in}$.
\end{theorem}

As a consequence, we deduce the following continuation criterion, which has appeared previously in \cite{GualdaniGuillen1} for radial solutions.

\begin{corollary}\label{thm:continuation}
    Fix $p > \frac{3}{2}$ and let $m >\frac{9}{2}\frac{p-1}{p - \frac{3}{2}}$ and $f_{in}$ an initial data satisfying the assumptions of Theorem \ref{thm:main}. Let $f$ be the corresponding solution constructed in Theorem \ref{thm:main}. Then, the maximal time $T^*$ for which $f$ can be uniquely continued as a smooth solution is characterized by
    \begin{equation}\label{eq:continuation-2}
        \lim_{t \nearrow T^*} \norm{f(t)}_{L^p} = +\infty.
    \end{equation}
\end{corollary}

The proof of Theorem \ref{thm:main} is based on the widely used lower bound $A[f] \gtrsim \brak{v}^{-3}$ that suggests that \eqref{eq:landau} should behave like the semi-linear heat equation  
\begin{equation*}
    \partial_t f = \Delta f + f^2.
\end{equation*}
Indeed, the smoothing estimate \eqref{eq:inst_reg} is consistent with the well-known $L^p(\R^3) \to L^\infty(\R^3)$ heat regularization estimate,
\begin{equation*}
    \norm{e^{t\Delta}f_{in}}_{L^\infty} \lesssim t^{-\frac{3}{2p}}\norm{f_{in}}_{L^p},
\end{equation*}
up to a correction accounting for the degenerate diffusion at large velocities encoded in the parameter $m$. Because the semi-linear equation is only locally well-posed for arbitrary initial data in $L^p$ for $p \ge 3/2$, it does not seem that one can go beyond the limitation $p > 3/2$ using solely the lower bound $A[f] \gtrsim \brak{v}^{-3}$ and not relying on the full strength of the diffusion. Indeed, \cite{GualdaniGuillen2} suggests that the heat smoothing estimate is suboptimal for Landau, but we cannot see this difference using methods relying on the lower bound $A[f] \gtrsim \brak{v}^{-3}$.

As a final consequence of our methods, we are able to adapt our a priori estimates to apply to rather general weak solutions to \eqref{eq:landau} satisfying a Prodi-Serrin type condition in \textit{unweighted} Lebesgue spaces. We prove a quantitative conditional regularity result which implies a more natural uniqueness class for the solutions constructed in Theorem \ref{thm:main}, namely $L^\infty(0,T;L^p)$, and implies the stability of these solutions under weak $L^r(0,T;L^p)$ convergence:
\begin{theorem}\label{thm:prodi_serrin}
Suppose $g:[0,T]\times \R^3 \to \R^+$ is an $H$-solution\footnote{More precisely, $g$ is a solution in the sense of distributions to the collisional form of the homogeneous Landau equation with constant mass, momentum, and energy, decreasing entropy, and bounded entropy production. As a consequence of \cite{Desvillettes_HSolutions}, the bounded entropy production implies $\nabla \sqrt{g} \in L^2(0,T;L^2_{-3})$.} to \eqref{eq:landau} such that
\begin{itemize}
    \item $g\in L^r(0, T; L^p(\R^3))$ for some $r > \frac{2p}{2p-3}$ and $p > 3/2$;
    \item $g\in W^{1,\infty}(\Omega)$ for each $\Omega$ compactly contained in $(0,T]\times \R^3$;
    \item and $\lim_{t \to 0^+} g(t) = g_{in}$ in the sense of distributions, where $g_{in} \in L^1_m \cap L\log L$ for some $m > m^*(p,r)$.
\end{itemize}
Then, $g \in L^1(0,T;L^\infty)$ and $g$ is smooth. Consequently, $g$ is the unique smooth solution to \eqref{eq:landau} in $L^r(0,T;L^p)$ with initial data $g_{in}$.
\end{theorem}

The role of the qualitative regularity of $g$, namely the locally Lipschitz assumption, is merely to guarantee that truncations of $g$ are sufficiently regular to be used as a test functions in a weak formulation of \eqref{eq:landau}. This assumption can be replaced by an appropriate notion of suitable weak solution, as in the seminal work of Caffarelli, Kohn, and Nirenberg on the partial regularity for the incompressible Navier-Stokes equations. However, we expect that in the $L^r(0,T;L^p)$ setting, this assumption is unnecessary and could potentially be removed by using a mollified version of $g$ as a test function and sending the mollification parameter to $0^+$ using the DiPerna-Lions commutator estimates, another tool frequently used in fluid mechanics. We do not investigate these directions further.

One can explicitly compute that $m^* = \frac{9}{2}\frac{p-1}{p-3/2}$ when $r = \infty$. As a result, we find that the solution constructed in Theorem \ref{thm:main} is the unique locally Lipschitz solution belonging to $L^\infty(0,T;L^p)$. This uniqueness class is the same one considered by Chern and Gualdani in \cite{ChernGualdani}, up to our additional qualitative regularity assumption. Although the estimates presented in \cite{ChernGualdani} are sound as a priori estimates, it remains unclear whether all low regularity solutions satisfy the same estimates in the absence of additional regularity hypotheses. Our assumption that g is locally Lipschitz is used in the proof of Theorem \ref{thm:prodi_serrin} to rigorously justify the uniqueness result of \cite{ChernGualdani}, which is especially intricate for $p \in (3/2, 2)$.


The term Prodi-Serrin condition originates in \cite{AlonsoBaglandDesvillettesLods_ProdiSerrin}, where the authors show the following conditional regularity result, analogous to the classical Prodi-Serrin result for the three-dimensional incompressible Navier-Stokes equation: If $f$ is a suitable solution to \eqref{eq:landau} with initial data in $L^1_m$ for some sufficiently large $m$ and
\begin{equation}\label{eq:weighted_assumption}
        \langle \cdot \rangle^{3} f \in L^r((0, T); L^p(\R^3)), \quad \frac{2}{r} + \frac{3}{p} = 2,\quad 1 \le r,\, p \le \infty, \quad p\neq \frac{3}{2},
\end{equation}
then $f\in L^\infty(t,T;L^\infty)$ for each $t > 0$ and, consequently, $f$ is instantaneously smooth. The core of their argument is a novel proof of an $\eps$-Poincar\'{e} inequality, which we use in Section \ref{sec:prodiserrin}. Note that except in the case $r = 1,\, p = \infty$, it remains unclear whether assumption \eqref{eq:weighted_assumption} is enough to guarantee uniqueness. The main contribution of Theorem \ref{thm:prodi_serrin} is that for the full range subcritical range of parameters, namely $\frac{2}{r} + \frac{3}{p} < 2$, one can rigorously prove that uniqueness holds as a consequence of smoothing estimates. We expect that uniqueness in the critical case, i.e. under assumption \eqref{eq:weighted_assumption}, requires a significantly different proof, because by analogy to the semi-linear heat equation, one expects only the estimate
\begin{equation*}
    \norm{f(t)}_{L^\infty} \le \frac{C}{t},
\end{equation*}
which does not imply $f \in L^1(0,T;L^\infty)$.

\subsection{Outline of the Paper}
In the next section, we state a few preliminary results, which will be used without proof. 
We then derive two a priori estimates: In Section \ref{sec:ode}, we use an ODE argument to show propagation of $L^p$ norms of the initial datum for $p > 3/2$ for some quantitative short time. In Section \ref{sec:degiorgi}, we use a new De Giorgi type argument to show quantitative $L^\infty$ regularization, for sufficiently integrable solutions. A combination of these estimates yields the quantitative short-time smoothing estimates claimed in \eqref{eq:inst_reg}. Then, in Section \ref{sec:proof}, we use a compactness argument and the estimates \eqref{eq:inst_reg} to build our smooth solutions and conclude the proof of Theorem \ref{thm:main} and Corollary \ref{thm:continuation}. Finally, in Section \ref{sec:prodiserrin}, we show how to use our a priori estimates to obtain the result in Theorem \ref{thm:prodi_serrin}.

\section{Preliminaries}\label{sec:prelims}

In this section, we collect some known results on the Landau equation that will be used in the proof of Theorem \ref{thm:main}. This includes prior results on local-in-time well-posedness, which will be used to build solutions, basic estimates on the coefficients $a[f]$ and $A[f]$ appearing in \eqref{eq:landau}, and results on the propogation of moments. Additionally, we include the key weighted Sobolev inequality that we use frequently below.

We begin with a local-in-time well-posedness result taken from \cite{HendersonSnelsonTarfulea} and used below. We present the result in a simplified form, suitable for our purposes:
\begin{theorem}[\protect{From\cite[Theorem 1.2 and Theorem 1.3]{HendersonSnelsonTarfulea}}]\label{thm:henderson_snelson_tarfulea}
Suppose $f_{in} \in \S(\R^3)$ is non-negative. Then, there is a time $0 < T^* \le \infty$ and a function $f:[0,T^*)\times \R^3 \rightarrow \R^+$ with $f\in C^\infty(0,T^*;\S(\R^3))$ such that $[0,T^*)$ is the maximal time interval on which $f$ is the unique classical solution to \eqref{eq:landau} with initial datum $f_{in}$. Furthermore, if $T^* < \infty$, then $f$ satisfies
\begin{equation*}
\lim_{t \nearrow T^*} \|f(t)\|_{L^\infty} = +\infty.
\end{equation*}
\end{theorem}

The following lemma controls the degenerate diffusion associated to \eqref{eq:landau} and quantifies the parabolic nature associated to the Landau equation when written in the form \eqref{eq:landau}. Each of the upper and lower coefficient bounds stated here are now standard and can be found in many places in the literature.
\begin{lemma}[From \protect{\cite[Lemma 3.1]{BedrossianGualdaniSnelson}}, \protect{\cite[Lemma 2.1]{GoldingGualdaniZamponi}}, \protect{\cite[Lemma 3.2]{Silvestre}}]\label{lem:Alower}
Suppose $f\in L^1_2(\R^3)$ is non-negative, satisfies the normalization \eqref{eq:normalization}, and has finite Boltzmann entropy, $\int f\log(f) \le H$. Then, there is a constant $c_0 = c_0(H)$ such that $A[f]$ satisfies the pointwise bound,
\begin{equation*}
\abs{A[f]} \ge \frac{c_0}{1 + |v|^3}.
\end{equation*}
Furthermore, if $f\in L^p$ for $3/2 < p \le \infty$, then
\begin{equation*}
\|A[f]\|_{L^\infty} \le C\|f\|_{L^1}^{\frac{2}{3}\frac{p - \frac{3}{2}}{p - 1}}\|f\|_{L^p}^{\frac{1}{3} \frac{p}{p-1}} 
\end{equation*}
Moreover, if $3 < p \le \infty$, then 
\begin{equation*}
    \norm{\nabla a[f]}_{L^\infty} \le C\|f\|_{L^1}^{\frac{p-3}{3(p-1)}}\|f\|_{L^p}^{\frac{2p}{3(p - 1)}} \qquad \text{and} \qquad \norm{\nabla a[f]}_{L^3} \le C\norm{f}_{L^{3/2}} 
\end{equation*}
\end{lemma}

We will frequently use that \eqref{eq:landau} propagates $L^1$-moments of any order. More precisely, $L^1$-moments of any order $s > 2$ grow at most linearly in time.
\begin{lemma}[\protect{From \cite[Lemma 2.1]{CarrapatosoDesvillettesHe}}]\label{lem:moments}
Let $k > 2$. Fix a non-negative initial datum $f_{in} \in L^1_k \cap L\log L$ satisfying the normalization \eqref{eq:normalization}. Suppose $f:\R^+ \times \R^3 \rightarrow \R^+$ is any weak solution of \eqref{eq:landau} with initial datum $f_{in}$. 
Then, for each $t \ge 0$,
\begin{equation}\label{eq:momentbound}
\int_{\R^3}f(t,v)\brak{v}^k\dd v \leq C(1+t).
\end{equation}
\end{lemma}

The following weighted Sobolev inequality is crucial for our method. This inequality is proved in the appendix of \cite{GoldingGualdaniLoher}, adapting the extensive literature on weighted Sobolev inequalities (see for instance \cite{SawyerWheeden}).
\begin{lemma}\label{lem:poincare}
Suppose $g: \R^3 \rightarrow \R$ is Schwartz class and let $1 \le s \le 6$. Then, there are universal constants $C_1(s)$ and $C_2(s)$ such that
\begin{equation*}
\left(\int_{\R^3} |g|^6 \brak{v}^{-9}\;dv\right)^{1/3} \le C_1\int_{\R^3} |\nabla g(v)|^2 \brak{v}^{-3} \dd v + C_2\left(\int_{\R^3} \abs{g}^s\dd v\right)^{2/s}.
\end{equation*}
\end{lemma}

\section{Propagation of \texorpdfstring{$L^p$}{Lp} norms}\label{sec:ode}

In this section, we present an ODE argument that provides a quantitative bound on the time that a smooth solution to the homogeneous Landau equation remains in $L^p$, in terms of norms of the initial datum. This is a variant of the ODE argument used in \cite{GoldingGualdaniLoher} to quantify how long solutions with small initial data remain small. We note that a version of this a priori estimate specialized to the case $p = 2$ first appeared in \cite{AlexandreLiaoLin}.

\begin{proposition}\label{prop:ode}
Fix $p > \frac{3}{2}$, $m > \frac{9}{2}\frac{p-1}{p-3/2}$, and $M \in \R^+$. Let $f:[0,T] \times \R^+ \rightarrow \R^+$ be any smooth, rapidly decaying solution to \eqref{eq:landau} satisfying \eqref{eq:normalization} and having initial data $f_{in} \in L^p$ with
\begin{equation}
\int_{\R^3} \brak{v}^m f_{in} \dd v + \int_{\R^3} f_{in}\abs{\log(f_{in})} \dd v \le M.
\end{equation}
Then, there is a time $T_0 > 0$ depending only on $m$, $p$, $M$, and $\|f_{in}\|_{L^p}$ such that
\begin{equation}\label{eq:apriori}
\sup_{t\in (0,\min(T,T_0))} \|f(t)\|_{L^p}^p + \int_0^{\min(T,T_0)} \int_{\R^3}\brak{v}^{-3}\abs{\nabla f^{p/2}}^2 \dd v \dd t \le C,
\end{equation}
where $C$ depends only on $m$, $p$, $M$, and $\norm{f_{in}}_{L^p}$.
\end{proposition}

\begin{proof}
Multiplying \eqref{eq:landau} with $f^{p-1}$ and integrating in $v$, we integrate by parts and rearrange terms to obtain:
\begin{equation*}
\frac{1}{p}\frac{\dd}{\dd t} \int_{\R^3} f^p \dd v + \frac{4(p-1)}{p^2}\int_{\R^3} A[f]\nabla f^{p/2} \cdot \nabla f^{p/2} \dd v = \frac{p-1}{p}\int_{\R^3} f^{p+1}\dd v.
\end{equation*}
Using the lower bound on the diffusion coefficient from Lemma \ref{lem:Alower}, we have
\begin{equation}\label{eq:landau-test}
    \begin{aligned}
            \frac{\dd}{\dd t}\int_{\R^3} f^p \dd v + c_0\int_{\R^3} \abs{\nabla f^{\frac{p}{2}}}^2 \langle v \rangle^{-3}\dd v\leq C(p)\int_{\R^3} f^{p+1} \dd v.
    \end{aligned}
\end{equation}
To bound the right hand side, we first interpolate between $L^{3p}_{-9}$, $L^p$, and $L^1_m$ to find
\beq\label{eq:interpolation_estimate_1}
    \norm{f}_{L^{p+1}} \leq\norm{\langle \cdot \rangle^{-\frac{3}{p}}f}_{L^{3p}}^{\theta_1} \norm{ f}_{L^p}^{\theta_2}\norm{\langle\cdot\rangle^m f}_{L^1}^{\theta_3},
\eeq
provided $\theta_1,\ \theta_2,\ \theta_3 \in (0,1)$ and the following relations are satisfied: 
\beqs
    \theta_1 + \theta_2 + \theta_3 = 1, \qquad \frac{\theta_1}{3p} + \frac{\theta_2}{p} + \theta_3 = \frac{1}{p+1}, \qquad \frac{-3}{p}\theta_1 + m \theta_3 = 0.
\eeqs
We solve this system of constraints in terms of $m$ and $p$ as:
\begin{equation}\label{eq:thetas}
    \theta_1 = \frac{3m p}{(p+1)(2pm - 9p +9)}, \qquad \theta_2 = \frac{2p^2 m - 9p^2 -m p}{(p+1)(2pm - 9p +9)},\qquad \theta_3 = \frac{9}{(p+1)(2pm - 9p +9)},
\end{equation}
where we note that $\theta_1,\,\theta_2,$ and $\theta_3$ belong to $(0,1)$ for $p > \frac{3}{2}$ and $m > \frac{9p}{2p-1}$.
Applying the weighted Sobolev inequality from Lemma \ref{lem:poincare} gives
\bals
    \norm{f}^{p+1}_{L^{p+1}} \leq \left(\int_{\R^3} \brak{v}^{-3}\abs{\nabla f^{\frac{p}{2}}}^2 \dd v + \int_{\R^3} f^p \dd v\right)^{\frac{(p+1)\theta_1}{p}} \norm{ f}_{L^p}^{(p+1)\theta_2}\norm{\langle\cdot\rangle^m f}_{L^1}^{(p+1)\theta_3},
\eals
where $\frac{(p+1)\theta_1}{p} < 1$, using the restriction that $m > \frac{9}{2}\frac{p-1}{p-3/2} > \frac{9p}{2p-1}$. Therefore, an application of Young's inequality yields
\bal\label{eq:landau-test1}
    \norm{f}^{p+1}_{L^{p+1}} \leq  \frac{c_0}{2}\left(\int_{\R^3} \brak{v}^{-3}\abs{\nabla f^{\frac{p}{2}}}^2 \dd v + \int_{\R^3} f^p \dd v\right) + C(c_0,m,p)\left(\norm{ f}_{L^p}^{\theta_2}\norm{\langle\cdot\rangle^m f}_{L^1}^{\theta_3}\right)^{\frac{p(p+1)}{p-(p+1)\theta_1}}.
\eal
Therefore, setting $y(t) = \|f(t)\|_{L^p}^p$, we arrive at the differential inequality
\begin{equation}\label{eqn:ODE1}
    \frac{dy}{dt} + \frac{c_0}{2}\int_{\R^3} \brak{v}^{-3}\abs{\nabla f^{p/2}}^2 \le C(p, c_0) y(t) + C(p,c_0,m)y^{\alpha_1}(t)\left(\int_{\R^3} f\brak{v}^m \dd v\right)^{\alpha_2},
\end{equation}
where $\alpha_1 = \frac{(p+1)\theta_2}{p - (p+1)\theta_1} > 1$ and $\alpha_2 = \frac{p(p+1)\theta_3}{p - (p+1)\theta_1}$.
Since $\|f(t)\|_{L^1_m}$ grows at most linearly in time by Lemma \ref{lem:moments}, this simplifies to 
\begin{equation*}
    \frac{dy}{dt} \le C(p, c_0) y + C(p,c_0,m)y^{\alpha_1}(1+t)^{\alpha_2} =: C_1y + g(t)y^{\alpha_1}.
\end{equation*}
In the case of equality, this is a Bernoulli type ODE, motivating the substitution $v = \frac{y^{1-\alpha_1}}{1-{\alpha_1}}$, which gives
\begin{equation*}
    \frac{\dd v}{\dd t} \le C_1(1-\alpha_1)v + g(t).
\end{equation*}
Performing a second substitution of $w = e^{C_1(\alpha_1 -1)t}v$ yields
\begin{equation*}
    \frac{\dd w}{\dd t} \leq e^{C_1(\alpha_1-1)t} g(t).
\end{equation*}
Thus, integrating in time, 
\begin{equation}\label{eq:diff-ineq}
    e^{C_1(\alpha_1-1)t}v(t) = w(t) \le  w(0) + \int_0^t e^{C_1(\alpha_1-1)\tau} g(\tau) \dd\tau = v(0) + \int_0^t e^{C_1(\alpha_1-1)\tau} g(\tau) \dd\tau.
\end{equation}
We conclude that for $0 < t < T$, $y$ satisfies:
\begin{equation*}
    y(t) \le e^{C_1t}\left[y_0^{1-\alpha_1} + (1-\alpha_1)\int_0^t e^{-C_1(1-\alpha_1)\tau} g(\tau)\dd \tau\right]^{1/(1-\alpha_1)}.
\end{equation*}
Choosing $2T_0$ sufficiently small, depending only on $M$, $m$, and $p$, we can enforce 
\begin{equation}\sup_{0 < t < \min(T,2T_0)} y(t) \le Cy_0\end{equation} for some constant $C > 0$ depending only on $m$, $p$, and $M$. Plugging this bound for $y(t)$ back into the right hand side of \eqref{eqn:ODE1}, we find the desired bound until the time $T_0$.

\end{proof}

\section{Quantitative Short Time Smoothing}\label{sec:degiorgi}

In this section, we present a quantitative regularization estimate for the $L^\infty$ norm of smooth solutions to \eqref{eq:landau}. More precisely, we show that for smooth solutions to \eqref{eq:landau}, the $L^\infty$ norm is controlled by the $L^{p+1}_{t,v}$ norm for $p > 3/2$. For solutions of \eqref{eq:landau}, this norm is controlled by the energy functional appearing in \eqref{eq:apriori} using interpolation and Lemma \ref{lem:poincare}. The proof relies on a De Giorgi type iteration, introduced in \cite{DeGiorgi} to study the regularity of weak solutions to elliptic equations. The specific version here is adapted from \cite{GoldingGualdaniLoher}, which followed \cite{AlonsoBaglandsDesvillettesLods} and used the energy functional $\mathcal{E}$ defined below in \eqref{eq:energy} as a control quantity. Here we refine the method and use certain space-time Lebesgue norms as control quantities. 

\begin{proposition}\label{prop:degiorgi}
Fix $p > \frac{3}{2}$, $m > \frac{9}{2}\frac{p-1}{p - \frac{3}{2}}$, $M \in \R^+$. Suppose $f_{in}$ is a non-negative Schwartz class function, satisfying the normalization \eqref{eq:normalization} and the bounds
\begin{equation*}
    \norm{f_{in}}_{L^1_m} \le M \qquad \text{and} \qquad \int_{\R^3} f_{in} \abs{\log(f_{in})} \;dv \le M.
\end{equation*}
Let $f:[0,T^*)\times \R^3\rightarrow \R^+$ be the unique Schwartz class solution to \eqref{eq:landau} with initial datum $f_{in}$ and $0 < T^* \le \infty$ its maximal time of existence. Then, for $\gamma$ and $\beta$, defined in terms of $m$ and $p$ via
\begin{equation}\label{eq:exponents}
\beta = \frac{2}{3} - \frac{3}{m} \qquad \text{and} \qquad \gamma = \frac{2(p - 3/2)}{3m}\left[m - \frac{9(p-1)}{2(p-3/2)}\right],
\end{equation}
there is a constant $C = C(M, m, p)$ such that for any $0 < t < \min(1,T^*)$, $f$ satisfies
\begin{equation*}
\norm{f(t)}_{L^\infty(\R^3)}  \le C\left(A_0^{\frac{\beta}{\gamma}} +A_0^{\frac{\beta}{1 + \gamma + \beta}}t^{-\frac{1+\beta}{1 + \gamma +\beta}}\right), \qquad \text{where} \qquad A_0 :=  \int_{0}^{t}\int_{\R^3} f^{p+1}\dd v \dd t.
\end{equation*}
\end{proposition}

\begin{flushleft}
{\bf \underline{Step 1: Introduction of Level Set Functions and Functionals}}
\end{flushleft}

For any fixed $\ell \in \R^+$, we consider the level set function, denoted by 
\bals
	f_\ell(t, v) := f(t, v) - \ell, \quad f_\ell^+(t, v) = \max\left(f_\ell(t, v), 0\right).
\eals
For any $\ell \ge 0$ and $0 \le T_1 \le T_2 \le \min(1,T^*)$, we define two functionals of $f_\ell^+$:
\begin{equation}\label{eq:energy}
\begin{aligned}
    \mathcal{E}_\ell(T_1,T_2) &= \sup_{T_1 < t < T_2} \int_{\R^3} (f_\ell^+)^p \dd v + \int_{T_1}^{T_2}\int_{\R^3} \brak{v}^{-3}\abs{\nabla (f_\ell^+)^{p/2}}^2\dd v\dd t,\\[7pt]
    \mathcal{A}_\ell(T_1,T_2) &= \int_{T_1}^{T_2} \int_{\R^3} (f_\ell^+)^{p+1} \dd v\dd t.
\end{aligned}
\end{equation}
The functional $\mathcal{E}_\ell$ is an energy functional associated to \eqref{eq:landau}, which appears when one attempts standard $L^p$ energy estimates. Note that we unconditionally control $\mathcal{E}(0,T_0)$ from the ODE argument in Proposition \ref{prop:ode}. The functional $\mathcal{A}_\ell$ is the ``control quantity'' appearing in the statement of Proposition \ref{prop:degiorgi}, which will be the iterated functional below.
Our goal now is to show the following two inequalities relating the two functionals: for any $0 \le T_1 < T_2 \le T_3 \le \min(1,T^*)$ and any $0 < k < \ell$,
\begin{align}
\mathcal{A}_{\ell}(T_1,T_3) &\le \frac{C(m,p,M)}{(\ell - k)^{\gamma}} \mathcal{E}_{k}(T_1, T_3)^{1+\beta},\label{eq:energy1}\\
\mathcal{E}_{\ell}(T_2,T_3) &\le C(m,p,M)\left[\frac{1}{(T_2 - T_1)(\ell - k)} + 1 + \frac{\ell}{(\ell - k)} + \frac{\ell^2}{(\ell - k)^2}\right] \mathcal{A}_{k}(T_1, T_3)\label{eq:energy2},
\end{align}
where $\beta$ and $\gamma$ are as in the statement of Proposition \ref{prop:degiorgi}.

\begin{flushleft}
{\bf \underline{Step 2: Proof of \eqref{eq:energy1}}}
\end{flushleft}

Let us show \eqref{eq:energy1}. This step does {\bf NOT} use that $f$ satisfies the Landau equation, but rather relies on functional inequalities. We start by noting that the weighted Sobolev inequality in Lemma \ref{lem:poincare} gives for any function $g:\R^3 \rightarrow \R$
\beqs
    \norm{g}_{L^r} \leq\norm{\langle \cdot \rangle^{-\frac{3}{p}}g}_{L^{3p}}^{\theta_1} \norm{ g}_{L^p}^{\theta_2}\norm{\langle\cdot\rangle^\alpha g}_{L^1}^{\theta_3},
\eeqs
provided $0 < \theta_1,\ \theta_2,\ \theta_3 < 1$ and the following relations are satisfied: 
\beqs
    \theta_1 + \theta_2 + \theta_3 = 1, \qquad \frac{\theta_1}{3p} + \frac{\theta_2}{p} + \theta_3 = \frac{1}{r}, \qquad \frac{-3}{p}\theta_1 + \alpha \theta_3 = 0.
\eeqs
Imposing further that $r\theta_1 = p$, we find
\bals
    r\theta_2 = \frac{p\left[r - p -\frac{2}{3}\right]}{p-1},\qquad r\theta_3 = \frac{5p-3r}{3(p-1)},\qquad \alpha = \frac{9(p-1)}{5p-3r}, \qquad r \in \left(p+\frac{2}{3}, \frac{5p}{3}\right).
\eals
The weighted Sobolev inequality of Lemma  \ref{lem:poincare} with $s = 2$ then implies for a constant depending only on $p$, $r$, and $\alpha$:
\beq \label{eqn:interpolation_estimate}
    \norm{g}^r_{L^r} \leq C\left(\norm{\langle \cdot \rangle^{-\frac{3}{2}}\nabla \left(g^{\frac{p}{2}}\right)}_{L^{2}}^{2} +\norm{g}_{L^p}^p\right) \norm{g}_{L^p}^{p\left(\frac{r - p -\frac{2}{3}}{p-1}\right)}\norm{\langle\cdot\rangle^\alpha g}_{L^1}^{\frac{5p-3r}{3(p-1)}}.
\eeq
Note that the admissible range for $r$, i.e. $(p+ 2/3, 5p/3)$ is non-degenerate if and only if $p > 1$ and, moreover, $p+1$ is within this range if and only if $p > 3/2$.
With the estimate \eqref{eqn:interpolation_estimate} at hand for general $g$ and $r$, let us return to our specific setting:
Let us fix $r$ by setting $\alpha = m$, or, in other words, define $r$ via the relation
\begin{equation*}
m = \frac{9(p-1)}{5p - 3r}  \qquad \text{or equivalently,} \qquad r = \frac{5p}{3} - \frac{3(p-1)}{m}.
\end{equation*}
With this definition of $r$, we see that $r$ is in the admissible range $(p+2/3, 5p/3)$ provided $m > 9/2$. However, since we need to bound the $L^{p+1}$ norm, we see that $r > p + 1$ if and only if
\begin{equation*}
m > \frac{9}{2} \frac{p- 1}{p - \frac{3}{2}}.
\end{equation*}
Since by assumption we have enough moments, we set $\gamma = r - (p + 1) > 0$. Therefore, applying \eqref{eqn:interpolation_estimate} to $f_k^+$, we find the pointwise-in-time estimate
\begin{equation*}
    \norm{f_k^+(t)}_{L^{p+1+\gamma}}^{p + 1 + \gamma} \le C\left(\int_{\R^3}\brak{v}^{-3}\abs{\nabla(f_k^+)^{\frac{p}{2}}}^2 \dd v +\norm{f_k^+}_{L^p}^p\right) \norm{f_k^+}_{L^p}^{p\left(\frac{\gamma + 1/3}{p-1}\right)}\norm{\langle\cdot\rangle^\alpha f_k^+}_{L^1}^{\frac{5p-3r}{3(p-1)}}.
\end{equation*}
A simple exponent computation shows $(r-p - 2/3)/(p-1) = 2/3 - 3/m = \beta$. So, integrating over $T_1 < t < T_3 \le \min(T^*, 1)$ and using Lemma \ref{lem:moments} to bound the $L^1$ moments of $f$, we obtain  
\begin{equation}\label{eq:energy1_proof}
\begin{aligned}
    &\norm{f_k^+}_{L^{p+1+\gamma}([T_1,T_3]\times \R^3)}^{p + 1 + \gamma} \\
    &\qquad\le C(M)\left(\int_{T_1}^{T_3}\int_{\R^3}\brak{v}^{-3}\nabla(f_k^+)^{\frac{p}{2}} \dd v\dd t\right)\norm{f_k^+}_{L^\infty(T_1,T_3;L^p)}^{p\beta} + C(M)\norm{f_k^+}_{L^\infty(T_1,T_3;L^p)}^{p(1 + \beta)}\\
        &\qquad\le C(M)\mathcal{E}_k(T_1,T_3)^{1 + \beta}.
\end{aligned}
\end{equation}

Finally, we pass from the lower level set $k$ to a higher lever set $\ell$: for any $0 \leq k < \ell$,
\beqs
	0 \leq f_\ell^+ \leq f_k^+,
\eeqs
and if $f_\ell > 0$, then $f \geq \ell > k$. In this case, $f_k = f_\ell + \ell - k$ and $\frac{f_k}{\ell-k} = \frac{f_\ell}{\ell-k} + 1 \geq 1$. Therefore,
\beqs
	 \mathbbm{1}_{\{f_\ell \geq 0\}} \leq \frac{f_k^+}{\ell-k}.
\eeqs
In particular, for any $\alpha > 0$,
\beqs
	 \mathbbm{1}_{\{f_\ell \geq 0\}} \leq \left(\frac{f_k^+}{\ell-k}\right)^\alpha.
\eeqs
Consequently, we deduce
\beq\label{eq:flfk}
	f_\ell^+ \leq (\ell-k)^{-\alpha} (f_k^+)^{1+\alpha}, \quad \text{for any }\alpha \geq 0.
\eeq
Combining \eqref{eq:energy1_proof} with \eqref{eq:flfk} where $\alpha = \gamma$, we obtain
\begin{equation*}
    \mathcal{A}_\ell(T_1,T_3) \le \frac{1}{(\ell - k)^\gamma}\int_{T_1}^{T_3}\int_{\R^3} (f_k^+)^{p+\gamma+1} \dd v\dd t \le \frac{C(M)}{(\ell - k)^{\gamma}}\mathcal{E}_k(T_1,T_3)^{1+\beta},
\end{equation*}
which completes the proof of \eqref{eq:energy1}.

\begin{flushleft}
{\bf \underline{Step 3: Proof of \eqref{eq:energy2}}}
\end{flushleft}

Let us now show \eqref{eq:energy2}. Contrary to the proof of \eqref{eq:energy1}, the proof of \eqref{eq:energy2} relies essentially on $f$ solving the Landau equation. We begin by testing the weak formulation of \eqref{eq:landau} with $(f_\ell^+)^{p-1}$. Since $\partial_t f_\ell^+ = \partial_t f \chi_{\{f \geq \ell\}}$ and $\nabla_v f_\ell^+ = \nabla_v f \chi_{\{f \geq \ell\}}$, we rearrange terms as in the proof of Proposition \ref{prop:ode} to obtain
\bals
	\frac{1}{p}\frac{\dd}{\dd t}\int_{\R^3} (f_\ell^+)^p \dd v + &\frac{4(p-1)}{p^2}\int_{\R^3} A[f]\nabla (f_\ell^+)^\frac{p}{2} \cdot \nabla (f_\ell^+)^\frac{p}{2} \dd v\\
	&= \frac{(p-1)}{p}\int_{\R^3} (f_\ell^+)^{p+1} \dd v +\frac{(2p-1)\ell}{p} \int_{\R^3} (f_\ell^+)^{p}\dd v + \ell^2\int_{\R^3}  (f_\ell^+)^{p-1}\dd v.
\eals
Using the lower bound on the coefficient $A[f]$ from Lemma \ref{lem:Alower} and discarding constants that depend only on $p$ and $M$ yields
\begin{equation}\label{eq:enestim}
\begin{aligned}
	\frac{\dd}{\dd t}\int_{\R^3} (f_\ell^+)^p\dd v + \int_{\R^3} \brak{v}^{-3}&\abs{\nabla (f_\ell^+)^{\frac{p}{2}}}^2\dd v\\
	&\lesssim_{p,M} \int_{\R^3} (f_\ell^+)^{p+1}\dd v + \ell\int_{\R^3} (f_\ell^+)^{p}\dd v + \ell^2\int_{\R^3}  (f_\ell^+)^{p-1}\dd v.
\end{aligned}
\end{equation}
Then, integrating \eqref{eq:enestim} over $(t_1, t_2)$ yields
\begin{equation*}
\begin{aligned}
	\int_{\R^3} (f_\ell^+(t_2))^p\dd v + \int_{t_1}^{t_2}\int_{\R^3} \brak{v}^{-3}\abs{\nabla (f_\ell^+)^{\frac{p}{2}}}^2\dd v &\lesssim_{p,M} \int_{\R^3} (f_\ell^+)^p(t_1) \dd v + \int_{t_1}^{t_2}\int_{\R^3} (f_\ell^+)^{p+1}\dd v \dd t \\
    &\qquad+ \ell\int_{t_1}^{t_2}\int_{\R^3} (f_\ell^+)^{p}\dd v \dd t + \ell^2\int_{t_1}^{t_2}\int_{\R^3}  (f_\ell^+)^{p-1}\dd v \dd t.
\end{aligned}
\end{equation*}
Finally, taking the the supremum over $t_2 \in [T_2, T_3]$, averaging over $t_1 \in [T_1, T_2]$, and using \eqref{eq:flfk} with $\alpha \in \set{0, 1, 2}$, we find
\bal\label{eq:enestim_int}
    \mathcal{E}_\ell(T_2,T_3) &= \sup_{s\in (T_2, T_3)} \int_{\R^3} (f_\ell^+(s))^p \dd v + \int_{T_2}^{T_3}\int_{\R^3} \brak{v}^{-3}\abs{\nabla (f_\ell^+)^{\frac{p}{2}}}^2\dd v\\
    &\lesssim_{M,p} \frac{1}{T_2 - T_1}\int_{T_1}^{T_2}\int_{\R^3} (f_\ell^+)^p(t) \dd v \dd t\\
    &\qquad+ \int_{T_1}^{T_3}\int_{\R^3} (f_\ell^+)^{p+1}\dd v \dd t + \ell\int_{T_1}^{T_3}\int_{\R^3} (f_\ell^+)^{p}\dd v \dd t + \ell^2\int_{T_1}^{T_3}\int_{\R^3}  (f_\ell^+)^{p-1}\dd v \dd t\\
    &\lesssim_{M,p} \left[\frac{1}{(T_2 - T_1)(\ell - k)} + 1 + \frac{\ell}{(\ell - k)} + \frac{\ell^2}{(\ell - k)^2}\right] \mathcal{A}_k(T_1, T_3),
\eal
which completes the proof of \eqref{eq:energy2}.

\begin{flushleft}
    {\bf \underline{Step 4: De Giorgi iteration}}
\end{flushleft}

We now finish the proof of Proposition \ref{prop:degiorgi} using an iteration procedure inspired by that of De Giorgi, introduced in \cite{DeGiorgi}. In particular, we will interweave the inequalities \eqref{eq:energy1} and \eqref{eq:energy2} passing to smaller dyadic space-time domains and higher level sets. To this end, we define the following iteration quantities: Let $0 < t < T = \min(T^*,1)$ be the fixed time in the statement of Proposition \ref{prop:degiorgi} and consider for $n \in \N$,
\beq\label{eq:levels}
	\ell_n = K\left(1 - 2^{-n}\right), \quad t_n = t\left(1 - 2^{-n}\right), \quad \text{and} \quad A_n = \mathcal{A}_{\ell_n}(t_n, T),
\eeq
where $K > 0$ is a parameter to be chosen appropriately. Indeed, the main remaining task is to find a value of $K$ for which $\lim_{n\rightarrow \infty} A_n = 0$. By the definition of the control quantity and the choice of iteration parameters, this would imply $f(\tau,v) \le K$ for almost every $t < \tau < T$ and $v\in\R^3$.

To this end, we first use \eqref{eq:energy1} to obtain
\begin{equation*}
    A_{n+2} \le \frac{C}{(\ell_{n+2} - \ell_{n+1})^{\gamma}}\mathcal{E}_{\ell_{n+1}}(t_{n+2},T)^{1 + \beta}.
\end{equation*}
Next, we use \eqref{eq:energy2} to obtain the recurrence
\begin{equation}\label{eq:recurrence}
\begin{aligned}
    A_{n+2} &\le \frac{C}{(\ell_{n+2} - \ell_{n+1})^{\gamma}}\left[\frac{1}{(\ell_{n+1} - \ell_n)(t_{n+2} - t_n)} + 1 + \frac{\ell_{n+1}}{(\ell_{n+1} - \ell_n)} + \frac{\ell_{n+1}^2}{(\ell_{n+1} - \ell_n)^2}\right]^{1+\beta} A_n^{1 + \beta}\\
    &\le \frac{C_12^{n(1 + \gamma + \beta)}A_n^{1+\beta}}{K^\gamma}\left[\frac{1}{Kt} + 1\right]^{1+\beta},
\end{aligned}
\end{equation}
where $C_1$ is now a fixed constant.
Because $\beta$ is positive, $A_n$ is supergeometric in $n$ and $A_n$ should decay fast enough to beat the geometrically-growing (in $n$) coefficients in the recurrence. To prove this, we employ a discrete barrier argument to find a quantitative estimate for $K$. We expect $A_n$ to decay exponentially, so we search for a choice of parameters $K$ and $Q > 0$ such that the sequence $A_n^*$ defined as
\beqs
    A_n^* := A_0 Q^{-n}, \quad \text{for } n \in \N
\eeqs
satisfies \eqref{eq:recurrence} with the reversed inequality, and so by induction remains larger than $A_n$. More precisely, we want to find a choice of $K$ and $A$ so that $A_n^*$ satisfies
\bal
    A_{n+2}^* \geq \frac{C_12^{n(\gamma + 1 + \beta)}(A_n^*)^{1+\beta}}{K^\gamma}\left[\frac{1}{tK}+ 1\right]^{1 + \beta}.
\label{eq:reverse}
\eal
Inserting our ansatz for $A_n^*$, \eqref{eq:reverse} holds if
\bals
    1 \geq \frac{C_1A_0^{\beta}Q^2 \left(2^{\gamma + 1 + \beta} Q^{-\beta}\right)^n}{K^\gamma}\left[\frac{1}{tK} + 1\right]^{1 + \beta}.
\eals
We now choose $Q$ such that 
\beqs
    2^{\gamma + 1 + \beta}Q^{-\beta} \leq 1 \qquad \text{or equivalently,}\qquad Q \geq 2^{\frac{\gamma + 1 + \beta}{\beta}}.
\eeqs
Then, for \eqref{eq:reverse} to hold we need
\bals
    1 \geq \left[\frac{C_1A_0^{\beta}Q^2}{K^\gamma}\right]^{\frac{1}{1+\beta}}\left[\frac{1}{tK} + 1\right].
\eals
Next, choosing $K$ so each of the terms is smaller than a half, the recurrence \eqref{eq:reverse} holds. More precisely, we choose $K$ as
\begin{equation} \label{eq:max_K}
\begin{aligned}
    K =  \max\left(\left(2^{1+\beta}C_1Q^2\right)^{\frac{1}{\gamma}}, \left(2^{1+\beta} C_1Q^2\right)^{\frac{1}{1+\beta + \gamma}} \right) \max \left\{A_0^{\frac{\beta}{\gamma}};  A_0^{\frac{\beta}{\gamma + 1 + \beta}}t^{-\frac{1 + \beta}{1 + \gamma + \beta}}\right\}.
\end{aligned}
\end{equation}
With the above choices of $K$ and $Q$, if $A_n \le A_n^*$, by construction
\begin{equation*}
    A_{n+2} \le \frac{C_12^{n(\gamma + 1 + \beta)}A_n^{1+\beta}}{K^\gamma}\left[\frac{1}{tK}+ 1\right]^{1 + \beta} \le \frac{C_12^{n(\gamma + 1 + \beta)}(A_n^*)^{1+\beta}}{K^\gamma}\left[\frac{1}{tK}+ 1\right]^{1 + \beta} \le A_{n+2}^*.
\end{equation*}
Since $A_0 = A_0^*$, it follows by induction that $A_n \leq A_n^*$ for $n$ even. Moreover, since $Q > 1$ we deduce
\beqs
    \lim_{n \to \infty} A_{2n} \le \lim_{n \to \infty} A_{2n}^* = 0.  
\eeqs
Recalling the definition of $A_n$ in \eqref{eq:levels}, the monotone convergence theorem implies 
\beqs
	\norm{f_K^+}_{L^{p+1}}^{p+1} = 0, \qquad \text{or equivalently,} \qquad \|f\|_{L^\infty((t,T)\times \R^3)} \le K.
\eeqs
Finally, from the choice of $K$, we have for all $v\in \mathbb{R}^3$ and $\tau \in(t,T)$
\begin{equation*}
    f(\tau,v) \le K \le C\left(A_0^{\frac{\beta}{\gamma}} +A_0^{\frac{\beta}{1 + \gamma + \beta}}t^{-\frac{1+\beta}{1 + \gamma +\beta}}\right).
\end{equation*}
Recalling the definitions of the exponents $\beta$ and $\gamma$ and performing some basic algebraic manipulations gives the simplified expressions found in \eqref{eq:exponents}, completing the proof of Proposition \ref{prop:degiorgi}.

\section{Short time existence}\label{sec:proof}
\subsection{Proof of Theorem \ref{thm:main}}

\begin{flushleft}
    \underline{{\bf Step 1: Approximation procedure}}
\end{flushleft}

Fix $p > 3/2$ and $m > \frac{9}{2}\frac{p-1}{p-3/2}$ as in the statement Theorem \ref{thm:main}. Take $f_{in} \ge 0$ any initial data satisfying the hypotheses of Theorem \ref{thm:main}, namely that $f_{in}$ is normalized according to \eqref{eq:normalization} and $f_{in}$ satisfies
\begin{equation*}
    \int_{\R^3} f_{in}|\log f_{in}| + f_{in}^p + \brak{v}^mf_{in} \dd v \le K < \infty.
\end{equation*}
Now, we approximate $f_{in}$ by a family of Schwartz class initial datum $\{f_{in}^{\eps}\}_{\eps > 0}$, satisfying
\begin{itemize}
    \item {\bf Regularity:} For each $\eps > 0$, $f_{in}^{\eps} \in \S(\R^3)$;
    \item {\bf Normalization:} For each $\eps > 0$, $f_{in}^{\eps} \ge 0$ and $f_{in}^{\eps}$ satisfies \eqref{eq:normalization}. 
    \item {\bf Uniform bounds:} For each $\eps > 0$,
    \begin{equation*}
        \int_{\R^3} f_{in}^{\eps}|\log f_{in}^{\eps}| + \abs{f_{in}^{\eps}}^p + \brak{v}^mf_{in}^{\eps} \dd v \le 2K
    \end{equation*}
    \item {\bf Convergence:} The family $\{f_{in}\}_{\eps>0}$ converges to $f_{in}$ as $\eps \rightarrow 0^+$ strongly in $L^p \cap L^1_m$ and pointwise almost everywhere.
\end{itemize}
 Consequently, by Theorem \ref{thm:henderson_snelson_tarfulea}, there exist local-in-time Schwartz class solutions $f_{\eps}: [0,T_{\eps}) \times \R^3 \rightarrow \R^+$ to \eqref{eq:landau}, with initial data $f_{in}^{\eps}$ where $T_{\eps}$, the maximal time of existence, is either infinite or satisfies
\begin{equation}\label{eq:continuation}
    \lim_{t\nearrow T_{\eps}} \norm{f_{\eps}(t)}_{L^\infty} = \infty.
\end{equation}
For the remainder of the proof, the parameters $p$ and $m$, the constant $K$, and the family $f_{\eps}$ are fixed.

\begin{flushleft}
    \underline{{\bf Step 2: Uniform estimates}}
\end{flushleft}

We first find a uniform lower bound on $T_{\eps}$. To this end, we apply Proposition \ref{prop:ode} and \ref{prop:degiorgi} to find uniform-in-$\eps$ estimates on the family $\set{f_{\eps}}$. Appealing to Proposition $\ref{prop:ode}$ with $M = 2K$, we find a uniform time $0 < T_0 \le 1$ such that
\begin{equation}\label{eq:ode-feps}
\sup_{0 < t < \min(T_0,T^*)} \norm{f_{\eps}(t)}_{L^p}^p + c_0\int_0^{\min(T_0,T_*)}\int_{\R^3} \brak{v}^{-3}\abs{\nabla f_{\eps}^p}\dd v \dd t \le C\norm{f_{in}^{\eps}}_{L^p}^p \lesssim_K 1,
\end{equation}
where both $C$ and $c_0$ are uniform in $\eps$. Second, note that using the same method as in the proof of Proposition \ref{prop:degiorgi}, i.e. the estimate \eqref{eq:energy1_proof} and \eqref{eq:ode-feps} yields
\begin{equation*}
    \norm{f_{\eps}}_{L^{p+1+\gamma}([0,t]\times \R^3)} \lesssim_K 1.
\end{equation*}
Therefore, we use Proposition \ref{prop:degiorgi} with $\tilde p = p + \gamma$, which implies for each $\eps > 0$,
\begin{equation}\label{defn:alpha}
    \norm{f_{\eps}(t)}_{L^\infty} \le C t^{-\frac{1+\beta}{1 + \beta + \tilde{\gamma}}}, \qquad \tilde{\gamma} = \frac{2(p + \gamma - 3/2)}{3m}\left[m - \frac{9(p + \gamma - 1)}{2(p + \gamma - 3/2)}\right].
\end{equation}
We define $\beta^*(p,m) = \beta^* = \frac{1 + \beta}{1 + \beta + \tilde{\gamma}}$, which clearly satisfies $\beta^* \in (0,1)$.
Therefore,
\begin{equation}\label{eq:Linfty}
    \norm{f_{\eps}(t)}_{L^\infty} \le Ct^{-\beta^*}, \qquad 0 < t < \min(T_0, T_{\eps}),
\end{equation}
where again $C$ is uniform in $\eps$. We note that if $T_{\eps} < T_{0}$ for some $\eps > 0$, then
\begin{equation*}
\lim_{t \nearrow T_{\eps}} \norm{f_{\eps}(t)}_{L^\infty} \le CT_{\eps}^{-\beta^*} < \infty,
\end{equation*}
which contradicts the continuation criteria \eqref{eq:continuation}. So, we have found our uniform lower bound $T_0 \le T_{\eps}$. We now want to construct our strong solution $f(t)$ on $[0, T_0]$. To this end, we would like uniform-in-$\eps$ H\"older regularity estimates to gain strong compactness. We begin with H\"older estimates on the coefficients. 

The following estimate is taken from \cite[Proposition 4.7]{GualdaniGuillen1}, where it is stated in a restricted setting. Nevertheless, the same proof implies the localized estimate we need. Before we state the result, let us mention that the notion of H\"older continuity that we work with respects the parabolic scaling of \eqref{eq:landau}: we define the distance between two points $(t, v), (s, w) \in \R_+ \times \R^3$ as
\begin{equation*}
    d\left((t, v), (s, w)\right) := \left(\abs{t-s} + \abs{v-w}^2\right)^{\frac{1}{2}}.
\end{equation*}
Then, H\"older spaces are defined using this notion of distance, that is $f \in C^{0,\alpha}(\R_+ \times \R^3)$ for some $\alpha \in (0, 1)$ if for any $(t, v), (s, w) \in \R_+ \times \R^3$ there holds
\begin{equation*}
    \abs{f(t, v) - f(s, w)} \leq C  d\left((t, v), (s, w)\right)^\alpha.
\end{equation*}
The constant $C$ is the H\"older seminorm $C = [f]_{C^{0,\alpha}(\R_+ \times \R^3)}$.
Higher order H\"older spaces are then introduced as usual. For more details, see \cite[p.7]{LadyzhenskayaUraltseva-parabolic}.
\begin{lemma}\label{lem:coefficient_holder}
    Suppose $g:[0,T] \times \R^3$ is a solution to \eqref{eq:landau} satisfying $$\norm{g}_{L^\infty([0,T]\times \R^3)} + \norm{g_{in}}_{L\log(L)} + \norm{g_{in}}_{L_2^1} \le L$$ and $g\in C^{k,\alpha}((0,T)\times B_R)$ for some $k\in \N$, $\alpha \in (0,1)$, and $R > 0$. Then, $A[g], \nabla a[g] \in C^{k,\alpha}((0,T)\times B_{R/2})$, depending only on $L,\,k,\,R$.
\end{lemma}

\begin{proof}
We prove first the case $k = 0$ for $A[g]$. That is, assume $g\in C^{0,\alpha}([0,T]\times B_R)$. Our goal is to show the same regularity for $A[g]$ on a smaller domain.

Recall that $A[g]$ is defined as convolution with the matrix-valued kernel $K(v) = \frac{\Pi(v)}{8\pi|v|}$. Let $0 \le \varphi \le 1$ a smooth function which is identically $1$ on $B_{R/4}$ and vanishes outside $B_{R/2}$. We decompose the kernel $K = K_1 + K_2$ where $K_1 = \varphi K$ is supported in $B_R$ and $K_2 = (1-\varphi)K \in C^\infty(\R^3)$. Note that $K_1\in L^1(\R^3)$ and $K_2 \in W^{2,\infty}(\R^3)$. Let us first estimate the localized part of the kernel. For $0 < t_1 < t_2 < T$ and $v,w\in B_{R/2}$,
    \begin{equation*}
    \begin{aligned}
        \Abs{K_1\ast g(t_1,v) - K_1 \ast g(t_2,w)} &\le \int_{\R^3} K_1(v^*)\Abs{g(t_1,v - v^*) - g(t_2,w - v^*)}\dd v^*\\
            &\le \norm{g}_{C^\alpha((0,T)\times B_R)} (|t_1 - t_2| + |w-v|^2)^{\alpha/2}\int_{B_R} K_1(v^*) \dd v^*.
    \end{aligned}
    \end{equation*}
    We conclude $K_1\ast g \in C^{0,\alpha}((0,T)\times B_{R/2})$. Next, we look at the smooth, non-localized part of the kernel:
    \begin{equation*}
            \norm{K_2 \ast g}_{W^{1,\infty}_v} \le \norm{\nabla K_2 \ast g}_{L^\infty} \le \norm{\nabla K_2}_{L^\infty}\norm{g}_{L^1_v},
    \end{equation*}
    which implies $K_2 \ast g$ is spatially H\"older continuous on $(0,T)\times \R^3$. For temporal regularity, we differentiate in time, use the equation, and integrate by parts:
    \begin{equation*}
    \begin{aligned}
        \partial_t (K_2 \ast g)(t,v) &= \int_{\R^3} K_2(v - v^*)\partial_t g(t,v^*) \dd v^* \\
            &= -\int_{\R^3} \nabla_{v^*} K_2(v - v^*) \cdot \big[A[g]\nabla_{v^*} g - \nabla_{v^*} a[g]g \big](v^*) \;dv^*.
    \end{aligned}
    \end{equation*}
    Integrating by parts once more, and using that $\sum_j \partial_j A_{ij} = - \partial_i a$, 
    we have
    \begin{equation*}
    \begin{aligned}
        \partial_t (K_2 \ast g)(t,v) = \int_{\R^3}\nabla_{v^*}^2 K_2(v - v^*) : A[g]g - 2\nabla_{v^*} K_2(v - v^*) \cdot \nabla_{v^*} a[g]g \dd v^*.
    \end{aligned}
    \end{equation*}
    Therefore, using Lemma \ref{lem:Alower}, conservation of mass, and $g\in L^\infty([0,T]\times \R^3)$, we estimate
    \begin{equation*}
    \begin{aligned}
        \abs{\partial_t (K_2 \ast g)} &\le \norm{\nabla^2 K_2}_{L^\infty(\R^3)}\norm{A[g]}_{L^\infty_{t,v}}\norm{g}_{L^\infty(0,T;L^1(\R^3))} + 2\norm{\nabla K_2}_{L^\infty(\R^3)}\norm{\nabla a[g]}_{L^\infty_{t,v}}\norm{g}_{L^\infty(0,T;L^1(\R^3))}\\
            &\le C\norm{K_2}_{W^{2,\infty}(\R^3)}\left(\norm{g}_{L^\infty(0,T;L^1)}^{5/3}\norm{g}_{L^\infty_{t,v}}^{1/3} + \norm{g}_{L^\infty(0,T;L^1)}^{4/3}\norm{g}_{L^\infty_{t,v}}^{2/3}\right) < \infty
    \end{aligned}
    \end{equation*}
    Combining our estimates, we conclude $K_2 \ast g \in W^{1,\infty}([0,T] \times \R^3)$. It follows that $A[g] \in C^\alpha_{loc}((0,T)\times B_{R/2})$, as desired. The same technique of splitting the kernel into a localized $L^1$ kernel and a non-localized $W^{2,\infty}$ kernel yields the corresponding result for $\nabla a[g]$.
    To prove the higher order estimates, i.e. when $k > 0$, we note that both $A[g]$ and $\nabla a[g]$ are convolution-type operators. Therefore, we may put an integral number of derivatives on $g$ and use the $k=0$ estimate to conclude.
\end{proof}

As a direct consequence of these local H\"older estimates on the coefficients, we may bootstrap the regularity of $f_{\eps}$, uniformly in $\eps$, using Schauder estimates. 
\begin{lemma}[Higher Regularity]\label{lem:higher_regularity}
The family $f_{\eps}: [0,T_0]\times \R^3\rightarrow \R^+$ constructed above in Step 1 satisfies for each $R > 0$, $0 < t < T_0$, $k \in \N$,
$$ \norm{f_{\eps}}_{C^k((t,T_0)\times B_R)} \le C(K,R,t,k),$$
uniformly in $\eps$.
\end{lemma}

\begin{proof}
    We will prove this by a bootstrap argument. Fix some $t$ and $R > 0$. Then, by Lemma \ref{lem:Alower}, the coefficients $A[f_{\eps}]$ and $\nabla a[f_{\eps}]$ satisfy the global upper bounds
    \begin{equation*}
    \begin{aligned}
        \abs{A[f_{\eps}]} \le C\norm{f_{\eps}}_{L^\infty(0,T;L^1)}^{2/3}\norm{f_{\eps}}_{L^\infty(0,T\times \R^3)}^{1/3} &\le C(K,T),\\
            \abs{\nabla a[f_{\eps}]} \le C\norm{f_{\eps}}_{L^\infty(0,T;L^1)}^{1/3}\norm{f_{\eps}}_{L^\infty(0,T\times \R^3)}^{2/3} &\le C(K,T),
    \end{aligned}
    \end{equation*}
    as well as the locally uniform coercivity estimate:
    \begin{equation*}
        A[f_{\eps}] \ge \frac{c_0}{\brak{v}^3} \ge c(K,R).
    \end{equation*}
    Therefore, we conclude \eqref{eq:landau} is a divergence-form parabolic equation with bounded measurable coefficients. By the parabolic version of the De Giorgi-Nash-Moser theorem (see \cite[Theorem 10.1]{LadyzhenskayaUraltseva-parabolic}), for each $0 < t \le T$ and $R > 0$, there is an $\alpha \in (0,1)$ such that
    \begin{equation*}
       f_{\eps}\in C^{\alpha}((t,T)\times B_{4R}).
    \end{equation*}
    Therefore, by Lemma \ref{lem:coefficient_holder}, $A[f_{\eps}],\,\nabla a[f_{\eps}]\in C^{\alpha}((t,T)\times B_{2R})$. Thus, by parabolic Schauder estimates (see \cite[Theorem 12.1]{LadyzhenskayaUraltseva-parabolic}), we conclude $f_{\eps}\in C^{1,\alpha}((t,T)\times B_R)$, which proves the claim with $k = 1$. Iterating this procedure, gaining one derivative at a time, we conclude $f_{\eps}\in C^{k}((t,T)\times B_R)$ for any $k > 0$.
\end{proof}

\begin{flushleft}
    \underline{{\bf Step 3: Compactness and properties of the limit}}
\end{flushleft}
As a consequence of the uniform estimates, \eqref{eq:Linfty}, Lemma \ref{lem:higher_regularity} and standard compactness results, we have a limit $f:[0,T_0] \times \R^3 \rightarrow \R$ such that $f_{\eps} \rightarrow f$ in the following senses:
\begin{itemize}
    \item Strongly in $C^k((t,T_0)\times B_R)$, for each $k \in \N$, $R > 0$ and each $0 < t < T_0$;
    \item Pointwise almost everywhere in $[0,T_0]\times \R^3$;
    \item Weak-starly in $L^\infty([0,T_0];L^p)$;
    \item Strongly in $L^1([0,T_0] \times \R^3)$,
\end{itemize}
where the last strong convergence follows from uniform integrability, pointwise convergence, and tightness. Note that the convergence immediately implies we retain $f \ge 0$, $f$ satisfies the normalization \eqref{eq:normalization}, 
and $f$ has decreasing entropy. 
Moreover, we can assume that $t^{-\beta^*}f_{\eps} \to t^{-\beta^*}f$ weak-starly in $L^\infty([0,T_0]\times \R^3)$. Therefore, by lower semi-continuity of the norm, we retain the desired short time smoothing estimates, \eqref{eq:Linfty} and, in particular, $f \in L^1(0,T_0;L^\infty)$. We retain also that $f\in C^\infty((t,T_0)\times \R^3)$ for each $0 < t < T_0$, i.e. that $f$ becomes instantaneously smooth.

\begin{flushleft}
    \underline{{\bf Step 4: Passing to the limit}}
\end{flushleft}

We now pass to the limit and show $f$ constructed in the previous step satisfies the Landau equation \eqref{eq:landau} on $(0,T_0)\times \R^3$. Since $f$ is smooth on this domain, the notion of solution is unambiguous and, in particular, it suffices to show $f$ is a distributional solution to \eqref{eq:landau}. We first show convergence of the non-local coefficients: By Lemma \ref{lem:Alower},
\begin{equation*}
    \big\Vert A[f - f_{\eps}]\big\Vert_{L^1(0,T_0;L^\infty(\R^3))} \le C\norm{f - f_{\eps}}_{L^1(0,T_0;L^1(\R^3))}^{2/3}\norm{f - f_{\eps}}_{L^1(0,T_0;L^\infty(\R^3))}^{1/3},
\end{equation*}
where the right hand side tends to $0$ as $\eps \to 0^+$, using the short-time smoothing estimates \eqref{eq:Linfty} and strong convergence in $L^1([0,T_0]\times \R^3)$. Using Lemma \ref{lem:Alower} again, we have the analogous estimate for $\nabla a[f - f_{\eps}]$:
\begin{equation*}
    \big\Vert\nabla a[f - f_{\eps}]\big\Vert_{L^1(0,T_0;L^\infty(\R^3))} \le C\norm{f - f_{\eps}}_{L^1(0,T_0;L^1(\R^3))}^{1/3}\norm{f - f_{\eps}}_{L^1(0,T_0;L^\infty(\R^3))}^{2/3},
\end{equation*}
which again converges to zero as $\eps \to 0^+$. Now, integrating by parts, $f_{\eps}$ satisfies the following distributional formulation of \eqref{eq:landau}: For any $\varphi \in C^\infty_c((0,T_0)\times \R^3)$,
\begin{equation*}
    \int_0^{T_0}\int_{\R^3} \varphi \partial_t f_{\eps} \dd v \dd t = -\int_0^{T_0}\int_{\R^3} \nabla \varphi \cdot \left(A[f_{\eps}]\nabla f_{\eps} - \nabla a[f_{\eps}]f_{\eps}\right) \dd v\dd t
\end{equation*}
Therefore, global convergence of the non-local coefficients and the local compactness in $C^1$ is sufficient to show that $f$ satisfies \eqref{eq:landau} in the following distributional sense: for each $\varphi \in C^\infty_c((0,T_0) \times \R^3)$,
\begin{equation*}
    \int_0^{T_0}\int_{\R^3} \partial_t f \varphi \dd v = - \int_0^{T_0}\int_{\R^3} \nabla \varphi \cdot \left(A[f]\nabla f - \nabla a[f]f\right) \dd v\dd t.
\end{equation*}
Because $f$ is smooth, we may integrate by parts and conclude $f$ is a classical solution to the Landau equation, i.e. $f$ solves \eqref{eq:landau} pointwise.

\begin{flushleft}
    \underline{{\bf Step 5: Behavior at initial time}}
\end{flushleft}

It only remains to show that the initial data is obtained in the appropriate strong sense.
To this end, we recall that each $f_{\eps}$ solves the following distributional form of \eqref{eq:landau}: For any $\varphi \in C^\infty_c([0,T_0)\times \R^3)$,
\begin{equation*}
    \int_0^{T_0}\int_{\R^3} \partial_t f_{\eps}\varphi \dd v = \int_0^{T_0}\int_{\R^3} \nabla^2 \varphi : A[f_{\eps}]f_{\eps} + 2\nabla \varphi \cdot \nabla a[f_{\eps}]f_{\eps} \dd v\dd t.
\end{equation*}
We estimate the first term on the right hand side using H\"older's inequality and Lemma \ref{lem:Alower} as
\begin{equation*}
\begin{aligned}
    \int_0^{T_0} \nabla^2 \varphi : A[f_{\eps}]f_{\eps} \dd v\dd t &\le \norm{\varphi}_{L^1\left(0,T_0;W^{2,\frac{p}{p-1}}\right)}\norm{f_{\eps}}_{L^\infty(0,T_0;L^p)}\norm{A[f_{\eps}]}_{L^\infty([0,T_0]\times \R^3)}\\
    &\le \norm{\varphi}_{L^1\left(0,T_0;W^{2,\frac{p}{p-1}}\right)}\norm{f_{\eps}}_{L^\infty(0,T_0;L^p)}^{1 + \frac{1}{3}\frac{p}{p-1}}\norm{f_{\eps}}_{L^\infty(0,T_0;L^1)}^{\frac{2}{3}\frac{p-3/2}{p-1}}\\
    &\le C(K)\norm{\varphi}_{L^1\left(0,T_0;W^{2,\frac{p}{p-1}}\right)}.
\end{aligned}
\end{equation*}
We then estimate the second term using H\"older's inequality, the Sobolev embedding $\dot{W}^{1,\frac{p}{p-1}}(\R^3) \embeds L^{\frac{2p-3}{3p}}$ for $p > 3/2$,  Lemma \ref{lem:Alower}, and the Hardy-Littlewood-Sobolev inequality as
\begin{equation*}
\begin{aligned}
    \int_0^{T_0} \nabla \varphi \cdot \nabla a[f_{\eps}]f_{\eps} \dd v\dd t &\le \norm{\nabla \varphi}_{L^1\left(0,T_0;L^{\frac{3p}{2p-3}}\right)}\norm{\nabla a[f_{\eps}]}_{L^\infty(0,T_0;L^3)}\norm{f_{\eps}}_{L^\infty(0,T_0;L^p)}\\
        &\le \norm{\nabla^2 \varphi}_{L^1\left(0,T_0;L^{\frac{p}{p-1}}\right)}\norm{f_{\eps}}_{L^\infty(0,T_0;L^{3/2})}\norm{f_{\eps}}_{L^\infty(0,T_0;L^p)}\\
        & \le C(K)\norm{\nabla^2 \varphi}_{L^1\left(0,T_0;L^{\frac{p}{p-1}}\right)}.
\end{aligned}
\end{equation*}
Therefore, by a duality argument, we conclude that
\begin{equation*}
    \norm{\partial_t f_{\eps}}_{L^\infty(0,T_0;W^{-2,p})} \le C(K).
\end{equation*}
Combined with $\norm{f_{\eps}}_{L^\infty(0,T_0;L^p)} \le C(K)$, the Aubin-Lions Lemma yields compactness of $f_{\eps}$ in $C(0,T_0;W^{-1,p})$. Therefore, up to a further subsequences we may assume $f_{\eps} \to f$ in $C(0,T_0;W^{-1,p})$. In particular, $f_{\eps}(0) = f_{in}^{\eps}$ implies $f$ is continuous in time and $f(0) = f_{in}$ in the sense of distributions. Using $f \in L^\infty(0,T_0;L^p)$ once more, $f$ is weakly continuous in time taking values in $L^p$ and $f(0) = f_{in}$ almost everywhere. By lower-semicontinuity of the norm in the weak topology, we also have
\begin{equation*}
    \norm{f_{in}}_{L^p} \le \liminf_{t\searrow 0^+} \norm{f(t)}_{L^p}.
\end{equation*}
Finally, performing a standard $L^p$ energy estimate on $f_{\eps}$ and throwing away the coercive term, we find
\begin{equation*}
    \frac{\dd}{\dd t}\int_{\R^3} f_{\eps}^p(t) \dd v \le C(p)\int_{\R^3} f_{\eps}^{p+1} \dd v \le C(p)\norm{f_{\eps}(t)}_{L^\infty}\norm{f_{\eps}}_{L^p}^p \le Ct^{-\beta^*},
\end{equation*}
where the constant $C$ is uniform in $\eps$. Therefore, integrating in time, for any $0 < t < T_0$,
\begin{equation*}
    \sup_{0 < s < t} \norm{f_{\eps}(s)}_{L^p_v} \le \norm{f^{\eps}_{in}}_{L^p_v} + Ct^{1-\beta^*}.
\end{equation*}
We take $\eps \to 0^+$ and use that $f_{in}^{\eps} \rightarrow f_{in}$ strongly in $L^p$ and $f_{\eps} \to f$ weak starly in $L^\infty(0,T;L^p)$ to obtain
\begin{equation*}
    \sup_{0 < s < t} \norm{f(s)}_{L^p} \le \norm{f_{in}}_{L^p} + Ct^{1-\beta^*}, \qquad \text{ for each }0 < t < T_0.
\end{equation*}
Since $f$ is weakly continuous taking values in $L^p$, for every $0 < t < T_0$,
\begin{equation*}
    \norm{f(t)}_{L^p} \le \norm{f_{in}}_{L^p} + Ct^{1-\beta^*}.
\end{equation*}
Finally, taking $t \to 0^+$ and using $\beta^* \in (0,1)$, we obtain
\begin{equation*}
\limsup_{t \searrow 0^+}\norm{f(t)}_{L^p} \le \norm{f_{in}}_{L^p}.
\end{equation*}
We have shown $\lim_{t\searrow 0^+} \norm{f(t)}_{L^p} = \norm{f_{in}}_{L^p}$. Combined with $f(t) \weak f_{in}$ in $L^p$, we deduce that $f(t) \to f_{in}$ strongly in $L^p$, which concludes the proof of Theorem \ref{thm:main}.

\subsection{Continuation Criterion}

In this part, we prove the continuation criterion contained in Theorem \ref{thm:continuation}. To this end, assume $f:[0,T)\times \R^3 \rightarrow \R^+$ is a solution constructed in the proof of Theorem \ref{thm:main}. We need only show that if $\limsup_{t \nearrow T}\norm{f(t)}_{L^p} < \infty$, then the solution $f$ may be continued smoothly past time $T$. Indeed, for each $0 < \eps < T$, set $g^{\eps}_{in} = f(T - \eps)$. Then, the propagation of moments from Lemma \ref{lem:moments} and the decay of entropy imply the a uniform-in-$\eps$ bound:
\begin{equation*}
    \int_{\R^3} g^{\eps}_{in}\abs{\log(g^{\eps}_{in})} + \abs{g^{\eps}_{in}}^p + \brak{v}^{m}g^{\eps}_{in} \dd v \le K.
\end{equation*}
Therefore, Theorem \ref{thm:main} and the invariance of \eqref{eq:landau} under translations in time implies there is a uniform time $2T_0$, such that $g^{\eps}_{in}$ gives rise to a smooth solution $g_{\eps}$ to \eqref{eq:landau} on $[T-\eps,T+2T_0 - \eps]$ with $g^{\eps}(T - \eps) = g^{\eps}_{in}$. Pick $\eps = T_0$ and set $\tilde f$ to be the concatenation of $g_{\eps}$ and $f$:
\begin{equation*}
    \tilde f(t) = \begin{cases} f(t) & 0 \le t \le T - T_0\\ g_{\eps}(t) &T - T_0 \le t \le T + T_0.    \end{cases}
\end{equation*}
Note that by construction $f(T - T_0) = g_{\eps}(T - T_0)$ and both $f$ and $g$ are solutions to \eqref{eq:landau} on $[T - T_0, T)$ that lie in $L^1(T-T_0, T;L^\infty)$ by the short time smoothing estimates of Theorem \ref{thm:main}. Therefore, by the uniqueness result of Fournier in \cite{Fournier}, $f(t) = g(t)$ on $[T - T_0,T)$ and $\tilde f$ is a smooth solution to \eqref{eq:landau} on $[0,T + T_0)$, as desired.

\section{Uniqueness}\label{sec:prodiserrin}

In this section, we prove Theorem \ref{thm:prodi_serrin} by extending the a priori estimates of the preceding sections to a more general class of solutions to \eqref{eq:landau}. In Section \ref{sec:ode} and Section \ref{sec:degiorgi}, we showed that solutions to \eqref{eq:landau} which are smooth, rapidly decay, and belong to $L^\infty(0,T;L^p)$ for some $p > 3/2$ are quantitatively bounded in $L^1(0,T;L^\infty)$. In this section, we extend this $L^1(0,T;L^\infty)$ bound to locally Lipschitz solutions $g\in L^r(0,T;L^p)$ for some $p > 3/2$ and $r > \frac{2p}{2p-3}$. The proof is divided into two steps:

In step one, mimicking Section \ref{sec:ode}, we use an ODE argument to quantitatively bound the potential growth of of the $L^p$ norm. More precisely, we show the estimate:
\begin{equation}\label{eq:E_bound}
    \sup_{t < \tau < T} \norm{g(t)}_{L^p}^p + \int_{t}^T\int_{\R^3} \brak{v}^{-3} \abs{\nabla g^{p/2}}^2 \dd v \dd \tau \le \frac{C}{t^{p/r}}, \qquad\text{for each } 0 < t < T.
\end{equation}
This gives a quantitative estimate on the $L^\infty(t,T;L^p)$ norm of $g$. We give first a formal proof of \eqref{eq:E_bound}, which is rigorous for rapidly decaying solutions. In step two, we use the methods from Section \ref{sec:degiorgi} to obtain an estimate on the potential growth of the $L^\infty$ norm of $g$. We show $\norm{g(t)}_{L^\infty} \le Ct^{-1 + \delta}$ for some small $\delta > 0$, which implies $g \in L^1(0,T;L^\infty)$. We conclude with a rigorous proof of \eqref{eq:E_bound} for merely locally Lipschitz solutions to \eqref{eq:landau}, which depends crucially on the so called $\eps$-Poincar\'{e} inequalities used to propagate $L^p$ norms for solutions to \eqref{eq:landau} (see \cite{GualdaniGuillen2,AlonsoBaglandDesvillettesLods_ProdiSerrin, AlonsoBaglandsDesvillettesLods}). The particular version used was recently introduced as \cite[Proposition 1.7]{AlonsoBaglandDesvillettesLods_ProdiSerrin}.

\begin{flushleft}
    \textbf{\underline{Step 1: Formal proof of \eqref{eq:E_bound}}}
\end{flushleft}

Let $g \in L^r(0,T; L^p) \cap W^{1,\infty}(\Omega)$ for each $\Omega$ compactly contained in $(0,T)\times \R^3)$ be a solution to \eqref{eq:landau} with initial datum $g_{in}$, attained in the sense of distributions. We assume that $g_{in}$ satisfies the bounds
\begin{equation*}
    \int_{\R^3} \brak{v}^m g_{in} \dd v \le M \qquad \text{and} \qquad \int_{\R^3} g_{in} \abs{\log(g_{in})} \dd v \le M,
\end{equation*}
for some $p > 3/2$, $r > \frac{2p}{2p-3}$, and $m > \frac{9}{2}\frac{p-1}{p-3/2}$. Following the a priori estimates in Section \ref{sec:ode}, we use $g^{p-1}$ as a test function\footnote{The local regularity of $g$ is not sufficient to justify using $g^{p-1}$ as a test function, which is the reason why this argument is only formal. One should instead use $\varphi_R g^{p-1}$ as a test function, where $\varphi_R$ is a smooth cutoff to $B_R$. Integrating in time, one obtains an integral equation for the truncated $L^p$ norm. Using $g \in L^r(0,T;L^p)$, one obtains a uniform-in-$R$ estimate on the truncated $L^p$ norm analagous to \eqref{eq:ODE_ProdiSerrin_integral_form}, from which \eqref{eq:norm_bound} and \eqref{eq:E_bound} follow by the monotone convergence theorem.} and derive a differential inequality for $y(t) := \norm{g(t)}_{L^p}^p$ of the form 
\begin{equation}\label{eq:ODE_ProdiSerrin}
    \frac{\dd y}{\dd t} + \int_{\R^3}\brak{v}^{-3}\abs{\nabla g^{p/2}}^2 \dd v \le Cy + Cy^{\alpha_1}, \qquad \text{for each } 0 < t < T,
\end{equation}
where the constant depends on $g$ only through $M$, $T$, $m$, and $p$. The exponent $\alpha_1$ is the same as in \eqref{eqn:ODE1}, which we recall is
\begin{equation*}
    \alpha_1 := \frac{(2p^2 - p)m - 9p^2}{(2p^2 - 3p)m - 9p^2 + 9p} \qquad \text{and} \qquad p\alpha_1  - p = \frac{2mp - 9p}{m(2p - 3) - 9p + 9}.
\end{equation*}
To avoid the significant regularity issues at time $t = 0$, we integrate over $[s, t]$, for $0 < s < t < T$, which yields
\begin{equation}\label{eq:ODE_ProdiSerrin_integral_form}
\begin{aligned}
    y(t) + \int_s^t \int_{\R^3}\brak{v}^{-3}\abs{\nabla g^{p/2}}^2 \dd v \dd \tau  &\leq  y(s) + C\int_s^t y(\tau) \dd \tau + C \int_s^t  \left(\int_{\R^3} g^{p}(\tau,v) \dd v\right)^{\alpha_1-1}y(\tau) \dd \tau.
\end{aligned}
\end{equation}
By Gr\"onwall's inequality, we obtain the quantitative bound
\begin{equation*}
   \sup_{s < t < T} y(t) \le y(s)\exp\left(C(T-s) + C\int_s^T\norm{g(\tau)}_{L^p}^{p\alpha_1 - p} \right).
\end{equation*}
A straightforward computation implies 
\begin{equation*}
    p\alpha_1  - p = \frac{2mp - 9p}{m(2p - 3) - 9p + 9} \qquad \text{and} \qquad \lim_{m\to \infty} p\alpha_1  - p = \frac{2p}{2p-3}. 
\end{equation*}
Since is $p\alpha_1 - p$ is decreasing as a function of $m$ provided $m > \frac{9}{2}\frac{p-1}{p-3/2}$, we may take $m$ even larger, depending only on $p$ and $r$, so that $\frac{2p}{2p-3} < p\alpha_1 - p < r$. 
For this choice of $m$,
\begin{equation*}
y(t) \le Cy(s) \exp\left(\int_0^T \norm{g(\tau)}_{L^p}^r \dd \tau \right) \le Cy(s) \qquad \text{for each }0 < s < t < T.
\end{equation*}
Raising to the power $r/p$ and averaging over $s\in[0,t]$, we obtain a quantitative bound on the appearance of the $L^p$ norm of $g$:
\begin{equation}\label{eq:norm_bound}
\norm{g(t)}_{L^p}^r \le \frac{C}{t}\int_0^t \norm{g(s)}_{L^p}^r \dd s \le \frac{C}{t}.
\end{equation}
Consequently, using \eqref{eq:norm_bound} to control the right hand side of \eqref{eq:ODE_ProdiSerrin_integral_form}, we deduce \eqref{eq:E_bound}:
\begin{equation}
    \sup_{t < \tau < T} \norm{g(t)}_{L^p}^p + \int_{t}^T\int_{\R^3} \brak{v}^{-3} \abs{\nabla g^{p/2}}^2 \dd v \dd \tau \le \frac{C}{t^{p/r}}, \qquad\text{for each } 0 < t < T.
\end{equation}

\begin{flushleft}
    \textbf{\underline{Step 2: $L^\infty(t,T;L^p) \to L^1(0,T;L^\infty)$ regularization estimates}}
\end{flushleft}
Our goal now is to use Proposition \ref{prop:degiorgi} to obtain a quantitative bound on the $L^\infty$ norm of $g$. To this end, we use interpolation as in \eqref{eqn:interpolation_estimate} to bound higher $L^q([0,T]\times \R^3)$ norms of $g$:
\begin{equation*}
    \norm{g}^q_{L^q} \leq C\left(\norm{\brak{\cdot}^{-\frac{3}{2}}\nabla g^{\frac{p}{2}}}_{L^2}^2 +\norm{g}_{L^p}^p\right) \norm{g}_{L^p}^{p\left(\frac{q - p -\frac{2}{3}}{p-1}\right)}\norm{\langle\cdot\rangle^\alpha g}_{L^1}^{\frac{5p-3q}{3(p-1)}}, \qquad \text{where }\alpha = \frac{9(p-1)}{5p-3q}.
\end{equation*}
Integrating in time and using \eqref{eq:E_bound}, we find for each $q \in \left(p,\frac{5p}{3}\right)$, for $m > \alpha(p,q)$, 
\begin{equation*}
    A_q(t) := \int_t^T\int_{\R^3} g^q \dd v \dd \tau \le C t^{-\frac{p}{r}\frac{q - 5/3}{p-1}}
\end{equation*}
Therefore, using the $L^q(0,T;L^q) \to L^\infty(t,T;L^\infty)$ smoothing estimate\footnote{Technically, Proposition \ref{prop:degiorgi} is stated for smooth, rapidly decaying solutions. Nevertheless, we claim that the estimate applies to $g$ without change. The qualitative regularity and decay is solely used to ensure that $(f-k)_+^{p-1}$ for constant $k \in \R^+$ can be used as a test function in \eqref{eq:landau} in the proof of the energy estimate \eqref{eq:energy2}. Crucially, \eqref{eq:energy2} can be easily justified for $g$ because $g\in W^{1,\infty}_{loc}$ and $g$ satisfies \eqref{eq:E_bound}.} in Proposition \ref{prop:degiorgi}, we obtain
\begin{equation}\label{eq:Linfty_bound}
\begin{aligned}
    \norm{f(t)}_{L^\infty} \le CA_q(t/2)^{\frac{\beta}{\gamma}} + \left(A_q(t/2)^{\beta}t^{-1 - \beta}\right)^{\frac{1}{1+\beta+\gamma}} &\le Ct^{-\frac{p\beta}{r\gamma}\frac{q - 5/3}{p-1}} + Ct^{-\frac{p}{r}\frac{q - 5/3}{p-1}\frac{\beta}{1+\beta+\gamma} - \frac{1 + \beta}{1+\beta+\gamma}}\\
        &:= Ct^{\omega_1} + Ct^{\omega_2}.
\end{aligned}
\end{equation}
where $\beta = \beta(m)$ and $\gamma = \gamma(q,m)$ are exponents given explicitly by the formulas:
\begin{equation*}
    \beta = \frac{2}{3} - \frac{3}{m} \qquad\text{and} \qquad \gamma = \frac{2(q-1)-3}{3m}\left[m - \frac{9(q-2)}{2(q-1)-3}\right].
\end{equation*}
To conclude $g\in L^1(0,T;L^\infty)$, it remains only to show that there is a choice of $q$ and $m$ for which the exponents $\omega_1$ and $\omega_2$ in \eqref{eq:Linfty_bound} are strictly greater than $-1$. Instead of computing directly, we take the limit as $q\to 5p/3$ and $m \to \infty$ in the exponents. In particular,
\begin{equation*}
    \lim_{m\to\infty} \beta = \frac{2}{3} \qquad \text{and} \qquad \lim_{m\to\infty,\, q\to 5p/3} \gamma = \frac{10p-15}{9}
\end{equation*}
and, consequently, since $r > \frac{2p}{2p-3}$,
\begin{equation*}
\begin{aligned}
        &\lim_{m\to\infty,\, q\to 5p/3} \omega_1 = -\frac{10p}{(10p-15)r} > -1\\
        &\lim_{m\to\infty,\, q\to 5p/3} \omega_2 = -\frac{\frac{10p}{9r} + \frac{5}{3}}{1 + \frac{2}{3} + \frac{10p - 15}{9}} = -\frac{10p + 15r}{10pr} > -1
\end{aligned}
\end{equation*}
We conclude that for $m$ sufficiently large (depending only on $p$ and $r$) we can use the $L^q(0,T;L^q) \to L^\infty(t,T;L^\infty)$ smoothing estimate from Proposition \ref{prop:degiorgi} with $q$ close enough to $5p/3$ (depending on $r$) so that $\omega_1 > -1$ and $\omega_2 > -1$. Thus, we conclude for $g_{in} \in L^1_m$ for $m$ sufficiently large, $g \in L^1(0,T; L^\infty)$. By the uniqueness result of Fournier \cite{Fournier}, we conclude any two locally Lipschitz solutions with the same initial data $g_{in}$ that both also belong to $L^r(0,T;L^p)$ are equal. \qedhere

We conclude the manuscript with a rigorous proof of \eqref{eq:E_bound}.

\begin{flushleft}
    \textbf{\underline{Rigorous Proof of \eqref{eq:E_bound}}}
\end{flushleft}

We take $g$ a locally Lipschitz weak solution to \eqref{eq:landau} as in the statement of Theorem \ref{thm:prodi_serrin}. Combining $g \in L^r(0,T;L^p)$ for $p > 3/2$ with Lemma \ref{lem:Alower} and interpolation of Lebesgue spaces, we conclude that the nonlocal coefficients satisfy the pointwise in time bounds:
\begin{equation*}
    \norm{A[f]}_{L^\infty} \le \norm{g}_{L^1}^{\frac{2p-3}{3p-3}}\norm{g}_{L^p}^{\frac{p}{3p-3}} \quad \text{and} \quad \norm{\nabla a[f]}_{L^3} \le \norm{f}_{L^{3/2}} \le \norm{f}_{L^1}^{\frac{2p-3}{3p-3}}\norm{f}_{L^p}^{\frac{p}{3(p-1)}}
\end{equation*}
Consequently,
\begin{equation*}
    A[g] \in L^{\frac{3r(p-1)}{p}}(0,T;L^\infty) \quad \text{and} \quad \nabla a[g] \in L^{\frac{3r(p-1)}{p}}(0,T;L^3)
\end{equation*}
Since $g \in W^{1,\infty}(\Omega)$ for each $\Omega$ compactly contained in $(0,T)\times \R^3$, a rather straightforward density argument implies for any $\varphi \in L^\infty(s,t;W^{1,\infty}(\R^3))$ with compact support $\Omega$,
\begin{equation}\label{eq:weak_formulation}
    \int_0^T\int_{\R^3} \varphi \partial_t g \dd v \dd \tau = - \int_0^T\int_{\R^3} \nabla \varphi \cdot \left(A[g]\nabla g - \nabla a[g] g\right) \dd v \dd \tau.
\end{equation}

For $R > 1$, we take $\varphi_R$ to be the cutoff function:
\begin{equation*}
    \varphi_R(v) = \begin{cases} 1 - R^{-(p+1)}(|v| - R)_+^{p+1} & \text{if }|v| < 3R/2 \\ R^{-(p+1)}(2R - |v|)_+^{p+1} &\text{otherwise.} \end{cases}
\end{equation*}
We note that $\varphi_R \in W^{1,\infty}(\R^3)$ satisfies the usual cutoff properties, namely, $0 \le \varphi_R \le 1$, $\varphi_R = 1$ on $B_R$, and $\varphi_R = 0$ on $\R^3 \setminus B_{2R}$. However, $\varphi_R$ satisfies the additional useful pointwise inequality:
\begin{equation}\label{eq:pointwise_test_function}
    \abs{\nabla \varphi_R}^2 \le CR^{-2(p+1)}\varphi_R^{\frac{2p}{p+1}}.
\end{equation}
Since $g$ is locally Lipshitz and $\varphi_R$ is compactly supported, $g^{p-1} \varphi_R^2 \chi_{[s,t]}$ an admissible test function in \eqref{eq:weak_formulation}. Therefore, the fundamental theorem of calculus for Lipschitz functions guarantees
\begin{equation*}
    \int_{\R^3} \varphi_R^2 g^p(t) \dd v - \int_{\R^3} \varphi_R^2 g^p(s) \dd v =  - p\int_0^T\int_{\R^3} \nabla \left(\varphi_R^2 g^{p-1} \right) \cdot \left(A[g]\nabla g - \nabla a[g] g\right) \dd v \dd \tau.
\end{equation*}
We rewrite the right hand side by rearranging powers of $g$ and integrating by parts as
\begin{align}
    RHS &= -\frac{4(p-1)}{p}\int_s^t\int_{\R^3} \varphi_R^2 \nabla g^{p/2} \cdot A[g] \nabla g^{p/2} \dd v \dd \tau + 2\int_s^t\int_{\R^3} g^{p/2}\varphi_R \nabla \varphi_R  \cdot A[g]\nabla g^{p/2} \dd v \dd\tau \\
        &\quad + \int_s^t\int_{\R^3} g^p \nabla \varphi_R^2 \cdot \nabla a[g] \dd v \dd\tau + (p-1)\int_s^t\int_{\R^3} \varphi_R^2 g^{p+1} \dd v\dd\tau := \sum_{i=1}^4 I_i.
\end{align}
The term $I_1$ is the coercive term. The error terms $I_2$ and $I_3$ coming from the cutoff $\varphi_R$ are bounded together using the identity $\nabla \cdot A[g] = \nabla a[g]$ as
\begin{equation}
\begin{aligned}
    I_3 &= \int_s^t\int_{\R^3} \nabla \varphi_R^2 \cdot \nabla \cdot (A[g] g^p) - A[g]\nabla g^p \dd v \dd\tau = -\int_s^t\int_{\R^3} \nabla^2 \varphi_R^2 : A[g]g^p \dd v \dd\tau  - I_2\\
        &\le -I_2 + \frac{1}{R^2}\int_s^t \norm{A[g(\tau)]}_{L^\infty} \norm{\varphi_R g(\tau)^{p/2}}_{L^2}^2 \dd\tau.
\end{aligned}
\end{equation}
Lastly, the main term $I_4$ presents a difficulty not seen in the global bounds anywhere else in the manuscript. The cutoff functions must be distributed correctly in order to close the argument. To this end, we bound using the $\eps$-Poincare lemma found as Proposition 1.7 in \cite{AlonsoBaglandDesvillettesLods_ProdiSerrin}. For the reader's convenience, we reproduce the proof here:
To bound $I_4$, we use H\"older's inequality, the fractional Sobolev embedding $\dot{H}^s(\R^3) \embeds L^{\frac{2q}{q-1}}(\R^3)$, the Gagliardo-Nirenberg-Sobolev inequality, and Young's inequality: for any $q > 3/2$ and any $\eps > 0$,
\begin{equation}
\begin{aligned}
    \int_{\R^3} \varphi_R^2g^{p+1} &\le \norm{\brak{v}^{3}g}_{L^q(B_{2R})}\norm{\varphi_Rg^{p/2}\brak{v}^{-3/2}}_{L^{\frac{2q}{q-1}}}^2\\
        &\le \norm{\brak{v}^{3}g}_{L^q(B_{2R})}\norm{\varphi_R g^{p/2}\brak{v}^{-3/2}}_{\dot{H}^s}^2\\
        &\le \norm{\brak{v}^{3}g}_{L^q(B_{2R})}\norm{\varphi_R g^{p/2}\brak{v}^{-3/2}}_{\dot{H}^1}^{2s}\norm{g^{p/2}\brak{v}^{-3/2}}_{L^2}^{2-2s}\\
        &\le \eps\norm{\varphi_R g^{p/2}\brak{v}^{-3/2}}_{\dot{H}^1}^2 + C(\eps)\norm{\brak{v}^{3}g}_{L^q(B_{2R})}^{\frac{1}{1-s}}\norm{\varphi_R g^{p/2}\brak{v}^{-3/2}}_{L^2}^2,
\end{aligned}
\end{equation}
where the relation $\frac{q-1}{2q} = \frac{1}{2} + \frac{s}{3}$ implies
\begin{equation}
    \int_{\R^3} \varphi_R^2 g^{p+1} \le \eps\int_{\R^3} \abs{\nabla (\varphi_R g^{p/2}\brak{v}^{-3/2})}^2 \dd v + C(\eps)\norm{\brak{v}^{3}g}_{L^q(B_{2R})}^{\frac{2q}{2q-3}} \int_{\R^3} \varphi_R^2 g^{p}\brak{v}^{-3} \dd v.
\end{equation}
By explicit computation, the pointwise estimate \eqref{eq:pointwise_test_function}, and Young's inequality,
\begin{equation}
\begin{aligned}
    \abs{\nabla \left( \varphi_R g^{p/2} \brak{v}^{-3/2}\right)}^2 &\le C\brak{v}^{-3}\varphi_R^2\abs{\nabla g^{p/2}}^2 + C\brak{v}^{-3}g^p \abs{\nabla \varphi_R}^2 +  C\brak{v}^{-5}\varphi_R^2 g^{p}\\
        &\le C\brak{v}^{-3}\varphi_R^2\abs{\nabla g^{p/2}}^2 + \varphi_R^2g^{p+1} + C\varphi_R^2 g^p + CR^{-2(p+1)^2}\brak{v}^{-3(p+1)}.
\end{aligned}
\end{equation}
Therefore, picking $\eps$ sufficiently small, and using the lower bound for $A[g]$ guaranteed by Lemma \ref{lem:Alower}, for any $s < t < T$ and for any $q > 3/2$,
\begin{equation}\label{eq:ODE_Prodi_Serrin_rigorous}
\begin{aligned}
    \int_{\R^3} \varphi_R^2 g^p(t) \dd v + c_0&\int_s^t\int_{\R^3} \brak{v}^{-3}\varphi_R^2 \abs{\nabla g^{\frac{p}{2}}}^2\dd v \dd \tau \leq \int_{\R^3} \varphi_R^2 g^p(s) \dd v\\
        &+ C\int_s^t\left( R^{-2(p+1)^2} + R^{-2}\norm{A[g]}_{L^\infty} + \norm{g\brak{v}^3}^{\frac{2q}{2q-3}}_{L^q(B_{2R})} \right) \left( \int_{\R^3}\varphi_R^2 g^p \dd v \right) \dd \tau.
\end{aligned}
\end{equation}
By Gr\"onwall's lemma applied to the function $\tau \mapsto \int_{\R^3} \varphi_R^2 g^p(\tau) \dd v$, we conclude the explicit bound: for any $0 < s < t < T$ and any $q > 3/2$,
\begin{equation}\label{eq:ODE_Prodi_Serrin_norm_rigorous}
    \int_{\R^3} \varphi_R^2 g(t)^p \dd v \le  \left(\int_{\R^3} \varphi_R^2 g^p(s) \dd v\right) \exp\left(C\int_0^T R^{-2(p+1)^2} + R^{-2}\norm{A[g]}_{L^\infty} + \norm{g\brak{v}^3}^{\frac{2q}{2q-3}}_{L^q(B_{2R})} \dd \tau \right).
\end{equation}
Since we have already shown $A[g] \in L^1(0,T;L^\infty)$, the monotone convergence theorem implies
\begin{equation}\label{eq:ODE_Prodi_Serrin_norm_rigorous2}
    \int_{\R^3} g(\tau)^p \dd v \le  \left(\int_{\R^3} g^p(s) \dd v\right) \exp\left(C\int_0^T \norm{g\brak{v}^3}^{\frac{2q}{2q-3}}_{L^q} \dd \tau \right).
\end{equation}
To conclude as in the formal proof given in Step 1, we need only show that $\brak{v}^{3}g\in L^{\frac{2q}{2q-3}}(0,T;L^q)$ for some $q > 3/2$. Because $(r,p)$ is a subcritical exponent pair--meaning $\frac{2}{r} + \frac{3}{p} < 2$--and $(\infty,1)$ is a supercritical pair--meaning $\frac{2}{\infty} + \frac{3}{1} > 2$--we can interpolate between $L^r(0,T;L^p)$ and $L^\infty(0,T;L^1_m)$ to find $(\frac{2q}{2q-3},q)$, a critical exponent pair. Interpolating between $L^p$ and $L^1_m$, for $3/2 < q < p$,
\begin{equation*}
    \norm{\brak{v}^3 g}_{L^q} \le \norm{\brak{v}^\frac{3}{\theta} g}_{L^1}^{\theta}\norm{g}_{L^p}^{1-\theta}, \quad \text{where} \quad \frac{\theta}{1} + \frac{1-\theta}{p} = \frac{1}{q}.
\end{equation*}
Solving for $\theta$, we find $\theta = \frac{p-q}{q(p-1)}$ and $1-\theta = \frac{p(q-1)}{q(p-1)}$, so that
\begin{equation}
\begin{aligned}
    \int_0^T\left(\int_{\R^3} \brak{v}^{3q} g^q \dd v\right)^{\frac{2}{2q-3}} \dd \tau &\le \int_0^T\left(\int_{\R^3} \brak{v}^{\frac{3q(p-1)}{p-q}} g \dd v \right)^{\frac{2(p-q)}{(p-1)(2q-3)}}\left(\int_{\R^3} g^p \dd v \right)^{\frac{2q-2}{(p-1)(2q-3)}} \dd \tau\\
    &\le C(T,p,q,r)\norm{g}_{L^\infty(0,T;L^1_m)}^{\frac{2(p-q)}{(p-1)(2q-3)}} \norm{g}_{L^r(0,T;L^p)}^r,
\end{aligned}
\end{equation}
provided $q$ is chosen so that $r > \frac{2(q-1)p}{(p-1)(2q-3)}$. If $r = \infty$ this constraint is satisfied for any $3/2 < q < p$. So, taking $q$ arbitrarily close to $3/2$, we recover the constraint on the number of moments, $m > \frac{9}{2}\frac{p-1}{p-3/2}$ as claimed. On the other hand, if $r < \infty$, taking the limit $q\to p^-$ the constraint becomes $r > \frac{2p}{2p-3}$. So, for $r < \infty$, we can always pick $q$ sufficiently close to $p$ and $m > \frac{3q(p-1)}{p-q}$ so that the constraint is satisfied. Returning to \eqref{eq:ODE_Prodi_Serrin_norm_rigorous2}, we obtain: for each $0 < s < t < T$,
\begin{equation*}
    \int_{\R^3} g(t)^p \dd v  \le  C(p,r,T,m,M,\norm{f}_{L^r(0,T;L^p)})\left(\int_{\R^3} g^p(s) \dd v\right).
\end{equation*}
As in Step 1, taking an average over $s \in (0,t)$, we obtain $\norm{g(t)}_{L^p}^p \le Ct^{-p/r}$. Returning to \eqref{eq:ODE_Prodi_Serrin_rigorous} and taking the limit as $R\to \infty$, we conclude \eqref{eq:E_bound}.

\bibliographystyle{plain}
\bibliography{landau}

\end{document}